\documentclass[11pt, a4paper]{article}
\usepackage{indentfirst}
\usepackage{amsfonts}
\usepackage{amssymb}
\usepackage{mathrsfs}
\usepackage{amsmath}
\usepackage{amsthm}
\usepackage{enumerate}
\usepackage{cite}
\usepackage{mathrsfs}
\allowdisplaybreaks
\usepackage{geometry}
\usepackage{ifpdf}
\ifpdf
\usepackage[colorlinks=true, linkcolor=blue, citecolor=red, final, backref=page, hyperindex]{hyperref}
\else
\usepackage[colorlinks, final, backref=page, hyperindex, hypertex]{hyperref}
\fi

\usepackage{txfonts}
\usepackage{indentfirst}
\usepackage{latexsym}
\usepackage[all]{xy}

\voffset=- 1.5cm
\def\<{\langle}
\def\>{\rangle}

\newtheorem{thm}{Theorem}[section]
\newtheorem{lem}[thm]{Lemma}

\newtheorem{prop}[thm]{Proposition}
\newtheorem{ex}[thm]{Example}

\theoremstyle{definition}
\newtheorem{defn}{Definition}[section]

\theoremstyle{remark}
\newtheorem{re}{Remark}[section]
\begin{document}
	\title{\bf Maurer-Cartan characterization and cohomology of compatible LieDer and AssDer pairs }
	\author{\bf  Basdouri Imed, Ghribi Salima, Sadraoui Mohamed Amin, Satyendra Kumar Mishra}
	\author{{
			\ Basdouri Imed $^{1}$
			\footnote { Corresponding author, E-mail: basdourimed@yahoo. fr}, 
			\ Ghribi Salima $^{2}$
			\footnote { Corresponding author, E-mail: ghribisalima@yahoo. fr}, 
			\ Sadraoui Mohamed Amin $^{3}$
			\footnote { Corresponding author, E-mail: aminsadrawi@gmail.com}
		}\\
		\\
		{\small 1. University of Gafsa, Faculty of Sciences Gafsa, 2112 Gafsa, Tunisia}\\
		{\small 2. University of Sfax, Faculty of Sciences Sfax, BP
			1171, 3038 Sfax, Tunisia} \\
		{\small 3. University of Sfax, Faculty of Sciences Sfax, BP
			1171, 3038 Sfax, Tunisia} 
	}
	\date{}
	\maketitle
	\begin{abstract}
		A LieDer pair (respectively, an AssDer pair) is a Lie algebra equipped with a derivation (respectively, an associative algebra equipped with a derivation). A couple of LieDer pair structures on a vector space are called Compatible LieDer pairs (respectively, compatible AssDer pairs) if any linear combination of the underlying structure maps is still a LieDer pair (respectively, AssDer pair) structure. In this paper, we study compatible AssDer pairs, compatible LieDer pairs, and their cohomologies. We also discuss about other compatible structures such as compatible dendriform algebras with derivations, compatible zinbiel algebras with derivations, and compatible pre-LieDer pairs. We describe a relationship amongst these compatible structures using specific tools like Rota-Baxter operators, endomorphism operators, and the commutator bracket.
	\end{abstract}

	\numberwithin{equation}{section}
	
	\tableofcontents
	\section{Introduction}
	The compatible structures are useful in various fields of Mathematics and mathematical Physics. In literature, there have been several studies on compatible algebraic structures, we refer to the articles \cite{bolsinov1992compatible,bolsinov2002compatible,reyman1988compatible,strohmayer2008operads} for compatible Lie algebras and \cite{liu2023maurer,odesskii2006algebraic,odesskii2008pairs,chtioui2021co} for compatible associative algebras.
	Compatible Lie algebras naturally correspond to compatible linear Poisson structures, which are related to bi-Hamiltonian structures. Furthermore, they are related to linear deformations of Lie algebras and classical Yang-Baxter equation \cite{golubchik2006compatible,nijenhuis1967deformations}. We also refer to some other interesting articles on the theory of compatible Lie algebras \cite{dotsenko2007character,panasyuk2014compatible,wu2015compatible}.
	
	\medskip
	
	\noindent The notion of compatible algebraic structures is not limited to Lie and associative algebras, the notion has been studied in the context of pre-Lie algebras \cite{liu2023deformations}, dendriform algebras \cite{liu2019compatible}, and Lie-Yamaguti algebras \cite{sania2023compatible}, to name a few. 
	
	\medskip
	
	\noindent In this article, we are interested in algebraic structures equipped with a derivation map and related compatible structures.
	A Lie algebra equipped with a derivation was studied in \cite{tang2019cohomologies} with the name `LieDer pair'. In the associative case, it has been introduced with the name `AssDer pair' in \cite{das2020extensions}. Here, we study compatible structures on LieDer pairs and AssDer pairs. We also discuss compatible structures in the case of Dendriform algebras, zinbiel algebras, and pre-Lie algebras equipped with a derivation. Finally, we explore relationship among these algebras. 
	
	\medskip
	
	\noindent A classical approach to study a mathematical object is to associate an invariant with the object. One such invariant is obtained by cohomology theories for different algebraic structures, for example see \cite{chevalley1948cohomology,gerstenhaber1963cohomology,harrison1962commutative,hochschild1945cohomology}. Deformation theory of both Lie and associative algebras are studied in terms of Chevalley-Eilenberg and Hochschild cohomology, we refer to \cite{nijenhuis1967deformations} and \cite{gerstenhaber1964deformation} for more details. 
	
	\medskip
	
	\noindent Recently, a cohomology theory for LieDer pairs and AssDer pairs was introduced in \cite{tang2019cohomologies, das2020extensions}. In the present work, we define a cohomology for compatible LieDer pairs as well as compatible AssDer pairs. 
	
	\medskip
	
	\noindent The section-wise description of the article is as follows: In Section \ref{preliminary}, we recall definitions related to associative and non-associative algebras. In Section \ref{section 3}, we define the compatible AssDer pairs and compatible LieDer pairs. We also define the compatible structures for DendriDer pairs, zinDer pairs, and pre-LieDer pairs. Section \ref{section 4} is devoted to the dendrification of various algebraic structures with derivation (Der pairs structures). In sequel, we explain the relationship between the above-mentioned compatible structures. In Section \ref{section 5} and \ref{section 6}, we introduce a cohomology for \textbf{compatible AssDer pairs} and \textbf{compatible LieDer pairs}, respectively. In this paper, all vector spaces are considered over a field $\mathbb{K}$ of characteristic $0$. All linear maps are $\mathbb{K}$-linear and tensor products are defined over $\mathbb{K}$ unless otherwise stated.

\section{Preliminary}\label{preliminary}
\noindent	In this section, we recall definitions related to associative and non-associative algebraic structures. 

\subsection{Associative algebra structures}	
\noindent We first recall the notion of associative algebras and related structures, which will be used throughout this article.
\begin{defn}
	An associative algebra $(A,\mu)$ is a vector space $A$ equipped with a linear map $\mu:A\otimes A \rightarrow A$ satisfying the following associativity condition 
	\begin{equation*}
		\mu(\mu(x,y),z)=\mu(x,\mu(y,z)), \quad \forall x,y,z \in A.
	\end{equation*}
\end{defn}

\paragraph*{Hochschild Cochain Complex and the Gerstenhaber Bracket:} 
In \cite{gerstenhaber1963cohomology}, M. Gerstenhaber constructed a graded Lie algebra structure on the graded space $C^{\ast}(A,A):=\oplus_{n\geq 0}C^{n}(A,A)$ with $C^{n}(A,A):=\mathrm{Hom}(A^{\otimes n}, A)$. He defined the graded Lie bracket (also called the Gerstenhaber bracket) as follows
\begin{equation*}
	[f,g]_G=f\circ g-(-1)^{pq} g \circ f,\quad \text{ with } f \in C^{p+1}(A,A) \text{ and } g\in C^{q+1}(A,A).
\end{equation*}
Where 
\begin{equation*}
	(f\circ g)(x_1,\ldots,x_{p+q+1})=\displaystyle\sum_{i=1}^{p+1}(-1)^{(i-1)q}f(x_1,\ldots,x_{i-1},g(x_i,\ldots,x_{i+p}),\ldots,x_{p+q+1}) \in C^{p+q+1}(A,A).
\end{equation*}
Recall that the previous bracket plays an important role in the construction of associative algebras as
mentioned in the following proposition.
\begin{prop}\cite{chtioui2021co}
	Let $A$ be a vector space and $\mu:A \otimes A \rightarrow A$ be a bilinear map on $A$. Then $\mu$ defines an associative algebraic structure on $A$ if and only if $[\mu,\mu]_G=0$.
\end{prop}
So it follows from the previous results that if $(A,\mu)$ is an associative algebra, then the coboundary map $d^n:C^n(A,A)\rightarrow C^{n+1}(A,A)$ given by 
\begin{equation*}
	d^n(\varphi)=(-1)^{n-1}[\mu,f]_G,\quad \text{for } f \in C^n(A,A),
\end{equation*}
is just the cohomology of $A$ with coefficients in itself.

\begin{defn} \cite{chtioui2021co}
	A compatible associative algebra is a triple $(A,\mu_1,\mu_2)$ where $(A,\mu_1)$ and $(A,\mu_2)$ are two associative algebras such that 
	\begin{equation}\label{comp associative algebras}	 
		\mu_2(\mu_1(x,y),z)+\mu_1(\mu_2(x,y),z)=\mu_1(x,\mu_2(y,z))+\mu_2(x,\mu_1(y,z)), \quad\text{ for all } x,y,z \in A.
	\end{equation}
\end{defn}

\begin{ex}\cite{chtioui2021co}
	Let $(A,\mu)$ be an associative algebra. A Nijenhuis operator on $A$ is a linear map $N:A\rightarrow A$ such that
	\begin{equation*}
		\mu(N(x),N(y))=N(\mu(N(x),y)+\mu(x,N(y))-N(x,y)),\quad \forall x,y\in A.
	\end{equation*}
	A Nijenhuis operator induces a new associative structure on $A$ denoted by 
	\begin{equation*}
		\mu_N:A\otimes L:\rightarrow A;\quad (x,y)\mapsto \mu_N(x,y),
	\end{equation*}
	and it is defined as follow 
	\begin{equation*}
		\mu_N(x,y)=\mu(N(x),y)+\mu(x,N(y))-N(x,y),\quad \forall x,y\in A.
	\end{equation*}
	Then it is easy to see that $(A,\mu,\mu_N)$ is a compatible associative algebras.
\end{ex}

\noindent Let $\delta :A \rightarrow A$ be a linear map and $(A,\mu)$ be an associative algebra, if $\delta$ satisfy the following identity 
\begin{equation*}
	\delta \mu(x,y)=\mu(\delta x,y)+\mu(x,\delta y), \quad\text{ for all } x,y \in A.
\end{equation*}
Then $\delta$ is called a derivation.	Let us recall the defintion of AssDer pairs from \cite{das2020extensions}.

\begin{defn}
	An AssDer pair $(A,\mu,\delta)$ is an associative algebra $(A,\mu)$ equipped with a derivation $\delta$.
\end{defn}
Let's recall cohomology of AssDer pairs, 

\paragraph*{Cochain Complex of AssDer pairs:}
In \cite{das2020extensions}, the authors introduced a cochain complex associated to AssDer pairs. For an AssDer pair $(A,\delta)$, the cochain complex is defined as follows: Define 
\begin{itemize}
	\item $C^0_{\mathrm{AssDer}}(A,A):=0$;
	\item $C^1_{\mathrm{AssDer}}(A,A):=\mathrm{Hom}(A,A)$;
	\item $C^n_{\mathrm{AssDer}}(A,A)=C^n(A,A)\oplus C^{n-1}(A,A)$, $n\geq 2$,
\end{itemize}
and the coboundary map $d^n_{\mathrm{AssDer}}:C^n_{\mathrm{AssDer}}(A,A)\rightarrow C^{n-1}_{\mathrm{AssDer}}(A,A)$ is given by 
$$
\left\{
\begin{array}{ll}
	d^1_{\mathrm{AssDer}}f=(d^1f,-Df), & \text{ for } f\in C^1_{\mathrm{AssDer}}(A,A),\\
	d^n_{\mathrm{AssDer}}(f_n,g_{n-1})=(d^nf_n,d^{n-1}g_{n-1}+(-1)^nDf_n), & \text{  for } (f_n,g_{n-1})\in C^n_{\mathrm{AssDer}}(A,A),
\end{array}
\right.
$$
where \begin{eqnarray*}
	Df=\displaystyle\sum_{i=1}^n f\circ(id\otimes \ldots \otimes \underbrace{\delta}_{\text{i-th place}}\otimes\ldots\otimes id)-\delta\circ f, \quad \text{for } f\in C^n(A,A).
\end{eqnarray*}
It can be seen, in terms of Gerstengaber bracket, as follow
\begin{equation*}
	Df=-[\delta,f]_G.
\end{equation*}
\subsection{Lie algebra structures}

\noindent Here, we recall some definitions related to Lie algebras, compatible Lie algebras, and LieDer pairs.

\begin{defn}
	A Lie algebra $(L,[\cdot,\cdot])$ is a vector space equipped with a skew-symmetric linear map $[\cdot,\cdot]:L\otimes L \rightarrow L$ such that the Jacobi identity is satisfied
	\begin{center}
		$\circlearrowleft_{x,y,z}[[x,y],z]=0, \quad \forall x,y,z \in L.$
	\end{center}
\end{defn}

A compatible Lie algebra is a couple of Lie algebra structures on a vector space such that any linear combination of these two Lie brackets yields a Lie bracket. Let us recall the following definition from \cite{golubchik2006compatible,golubchik2002compatible,golubchik2005factorization,liu2023maurer}.
\begin{defn} 
	A compatible Lie algebra is a triple $(L,[\cdot,\cdot]_1,[\cdot,\cdot]_2)$ consisting of a vector space $L$, two Lie algebra brackets $[\cdot,\cdot]_1$ and $[\cdot,\cdot]_2$ on $L$ such that
	\begin{equation*}
		\circlearrowleft_{x,y,z}[[x,y]_1,z]_2+\circlearrowleft_{x,y,z}[[x,y]_2,z]_1=0, \quad \forall  x,y,z \in L.
	\end{equation*}
\end{defn}
\begin{ex}
	Consider a two $2$-dimensional Lie algebras $(L,[\cdot,\cdot]_1)$ and $(L,[\cdot,\cdot]_2)$ with basis $\{e_1,e_2\}$ defined by 
	\begin{equation*}
		[e_1,e_2]_1=e_1 \text{ and } [e_1,e_2]_2=e_2.
	\end{equation*}
	Then $(L,[\cdot,\cdot]_1,[\cdot,\cdot]_2)$ is a compatible Lie algebra.
\end{ex}

In \cite{tang2019cohomologies}, Lie algebra with a derivation was studied with several canonical examples. These structures are called the LieDer pair.
\begin{defn}
	A LieDer pair is a triple $(L,[\cdot,\cdot],\delta)$ where $(L,[\cdot,\cdot])$ is a Lie algebra equipped with a derivation $\delta$, i.e.,
	\begin{eqnarray*}
		\delta([x,y])=[\delta(x),y]+[x,\delta(y)], \quad \forall~ x,y \in L.
	\end{eqnarray*}
\end{defn}

\paragraph*{Cohomology of LieDer pairs :}
Recall that from \cite{tang2019cohomologies}, the set of LieDer pair $n$-cochains is defined as follow:
\begin{itemize}
	\item The set of $0$-cochains LieDer pair is $0$,
	\item the set of LieDer pair $1$-cochain is defined as $\mathcal{C}_{\mathrm{LieDer}}^{1}(L;L)$ equal to $\mathrm{Hom}(L,L)$,
	\item for $n\geq 2$, the set of LieDer pair $n$-cochains is defined by 
	
	And the coboundary map $\mathcal{C}_{LieDer}^{n}(L;L)=\mathcal{C}^{n}(L;L)\times \mathcal{C}^{n-1}(L;L)$ is given by 
	$$
	\left\{
	\begin{array}{ll}
		\partial^1_{\mathrm{LieDer}} (f)=(d_{\mathrm{L}}^1f,-Df)& \text{ for } f\in C_{\mathrm{LieDer}}^1(L;L),\\
		\partial_{\mathrm{LieDer}}(f_{n},g_{n-1})=(d_{\mathrm{L}}^nf_{n}, d_{\mathrm{L}}^{n-1}g_{n-1}+ (-1)^{n}Df_{n}),& \text{ for } (f_{n},g_{n-1})\in C_{\mathrm{LieDer}}^n(L;L).	
	\end{array}
	\right.
	$$
	Where $D$ is an operator defined by $D:\mathcal{C}^{n}(L;L)\longrightarrow\mathcal{C}^{n}(L;L),$
	\begin{equation*}
		Df_{n}=\sum\limits_{i=1}^{n} f_{n}\circ(1\otimes \dots \otimes \underbrace{\delta}_{\text{i-th place}}\otimes\dots\otimes 1)-\delta\circ f_{n}.
	\end{equation*}
	
	and $d_{\mathrm{L}}^{\star}$ is the coboundary operator of Chevalley-Eilenberg of a Lie algebra  $L$.
\end{itemize}

\paragraph*{Maurer-Cartan Characterization of LieDer pairs : }
Tang Rang, Yael Fregier, and Yunhe Sheng, in their work cited as \cite{tang2019cohomologies}, introduced the concept of the LieDer pair and established that its structure constitutes a \textbf{Maurer-Cartan element}. Consequently, in the subsequent sections of this study, we will revisit various findings and notational conventions, which will prove instrumental in our subsequent discussions.\\

Expanding on the findings mentioned earlier, the findings presented in reference \cite{liu2023maurer}, when dealing with a Lie algebra denoted as $(L,[\cdot,\cdot])$, along with its associated \textbf{Maurer-Cartan element} represented by '$w$' they proposed a reformulation of the Chevalley-Eilenberg coboundary in the following manner:

\begin{equation*}
	d_{\mathrm{L}}^{n}:\mathcal{C}^{n}(L,L)\longrightarrow \mathcal{C}^{n+1}(L,L).
\end{equation*}   
Defined by 
\begin{equation*}
	d_{\mathrm{L}}^{n}f=(-1)^{n-1}[w,f], \ \ \ \forall  f \in\mathcal{C}^{n}(L,L).
\end{equation*}
Consider a vector space $L$  and let be $\vartheta$ a graded vector space with the inclusion of the suspension operator $s:\vartheta \rightarrow \vartheta$ which is further elaborated upon in the work as referenced \cite{tang2019cohomologies}. Define the graded vector space $DC^{\star}(L;L):=C^{\star}(L;L) \oplus sC^{\star}(L;L)=\displaystyle \oplus_{n\geq 0}(\mathrm{Hom}(\wedge^{n+1}L,L)\times \mathrm{Hom}(\wedge^{n}L,L))$. \\
Define a skew-symmetric bracket operation \\
$$\{\cdot,\cdot\}:DC^m(L;L)\otimes DC^n(L;L) \rightarrow DC^{m+n}(L;L) $$ by  \\
$$\{(f_{m+1,g_m}),(f_{n+1,g_n})\}:=([f_{m+1},f_{n+1}]_{\mathrm{NR}},(-1)^{m}[f_{m+1},g_n]_{\mathrm{NR}}-(-1)^{n(m+1)}[f_{n+1},g_m]_{\mathrm{NR}})$$
where $[\cdot,\cdot]_{\mathrm{NR}}$ is the Nijenhuis-Richardson bracket. 
\begin{thm} \cite{tang2019cohomologies}
	Utilizing the notations provided earlier, it can be established that the pair $(DC^{\star}(L,L),\{\cdot,\cdot\})$ constitutes a graded Lie algebra. Additionally, the LieDer pair structure on $L$ precisely corresponds to the \textbf{ Maurer-Cartan element} within this algebra.
\end{thm}
\begin{re}
	The previous theorem introduce that $(w,\delta)$ is a \textbf{ Maurer-Cartan element} if and only if $(L,w,\delta)$ is a LieDer pair.
\end{re}
Furthermore, in the same source, the authors define: 
\begin{equation*}
	d_{(w,\delta)}:DC^n(L,L) \rightarrow DC^{n+1}(L,L)  
\end{equation*}
by 
\begin{equation*}
	d_{(w,\delta)}(f_{n+1},g_n):=\Big\{(w,\delta),(f_{n+1},g_n) \Big\}
\end{equation*}
is a graded derivation of the graded Lie algebra $(DC^{\star}(L,L),\{\cdot,\cdot\})$ such that $d_{(w,\delta)} \circ d_{(w,\delta)}=0$
\begin{lem} \cite{tang2019cohomologies}
	Let $(L,w,\delta)$ be a LieDer pair. Then $(DC^{\star}(L,L),\{\cdot,\cdot\},d_{(w,\delta)})$ is a differential graded Lie algebra.
\end{lem}
Finally, we proceed to revisit the following theorem and proposition as the last results we recall:
\begin{thm}
	Let $(L,w,\delta)$ be a LieDer pair. Then for two linear maps $w_1:L\wedge L \rightarrow L$ and $\delta_1:L \rightarrow L$, $(w+w_1,\delta+\delta_1)$ is a LieDer pair if and only if $(w_1,\delta_1)$ is a \textbf{ Maurer-Cartan element} of the differential graded Lie algebra   $(DC^{\star}(L,L),\{\cdot,\cdot\},d_{(w,\delta)})$ which mean 
	\begin{equation*}
		d_{(w,\delta)}(w_1,\delta_1)+\frac{1}{2}\{(w_1,\delta_1),(w_1,\delta_1)\}=0
	\end{equation*}
\end{thm}
\begin{prop}{\cite{tang2019cohomologies}}
	Let $(L,w,\delta)$ be a LieDer pair. Then we have for $(f_n,g_{n-1}) \in \mathrm{Hom}(\wedge^nL,L)\times \mathrm{Hom}(\wedge^{n-1}l,L)$
	\begin{equation*}
		\partial(f_n,g_{n-1})=(-1)^{n-1}d_{(w,\delta)}(f_n,g_{n-1}),
	\end{equation*}
	Where
	\begin{equation*}
		d_{(w,\delta)}(f_n,g_{n-1})=\{([w,f_n]_{NR},-[w,g_{n-1}]_{NR}-[f_n,\delta]_{NR})\}
	\end{equation*} 
\end{prop}	
\subsection{Some other non-associative algebra structures}
Within this subsection, we revisit several non-associative algebra structures, which encompass pre-Lie algebras, zinbiel algebras, and associated constructs. 

\begin{defn}\cite{burde2006left}
	A pre-Lie algebra is a pair $(L,\circ)$, where $L$ is a vector space equipped with a linear map $\circ:L\otimes L \rightarrow L$ satisfying
	\begin{equation*}
		(x\circ y)\circ z-x\circ(y\circ z)=(y\circ x)\circ z-y\circ(x\circ z), \quad \forall ~ x,y,z \in L.
	\end{equation*}
\end{defn}
We should remember that it is possible to construct a Lie algebra from a pre-Lie algebra in the following manner: Let $(L,\circ)$ be a pre-Lie algebra, then we consider the commutator
\begin{equation*}
	[x,y]_{\mathrm{c}}=x\circ y-y\circ x,\quad  \forall ~x,y \in L.
\end{equation*}
Consequently, the bracket denoted as $[\cdot,\cdot]_{\mathrm{c}}$ satisfies the Jacobi identity. This means that the pair $(L,[\cdot,\cdot]_{\mathrm{c}})$  forms a Lie algebra, referred to as the \textbf{sub-adjacent} Lie algebra of $(L,\circ)$. \\

\begin{defn}
	Let $(L,[\cdot,\cdot])$ be a Lie algebra over the field $\mathbb{K}$. A linear operator $R:L \rightarrow L$ is called a Rota-Baxter operator of weight $\lambda \in \mathbb{K}$ if $R$ satisfies the following relation
	\begin{equation*}
		[R(x),R(y)]+\lambda R([x,y])=R\Big([R(x),y]+[x,R(y)] \Big),\quad \forall x,y \in L.
	\end{equation*}
\end{defn}
Conversely, given a Rota-Baxter operator $R$ on a Lie algebra $(L,[\cdot,\cdot])$, we can construct a pre-Lie algebra product as follows
$$x \circ y=[R(x),y], \quad\forall ~x,y\in L.$$

For a deeper insight into these constructions,  we recommend consulting the following references: \cite{an2007rota,loday2001dialgebras,golubchik2000generalized}.

Recall that a linear map $\delta:L \rightarrow L$ is called a derivation on the pre-Lie algebra $(L,\circ)$ if it satisfies the following equation \begin{equation*}
	\delta (x \circ y)=\delta x \circ y+x \circ \delta y,\quad\forall x,y\in L. 
\end{equation*}

The triple $(L,\circ,\delta)$ is referred to as a 'pre-LieDer pair.' For further elaboration on this concept, refer to \cite{liu2023cohomologies}.

\begin{defn} \cite{liu2023deformations}
	Let $L$ be a vector space and let be $\circ_1$ and $\circ_2$ be two bilinear maps from $L \otimes L \rightarrow L$. Then the triple $(L,\circ_1,\circ_2)$ is said to be a compatible pre-Lie algebra if the following equation is satisfied 
	\begin{equation*}
		(x\circ_2y)\circ_1z+(x\circ_1y)\circ_2z-x\circ_1(y\circ_2z)-x\circ_2(y\circ_1z)-(y\circ_2x)\circ_1z-(y\circ_1x)\circ_2z+y\circ_1(x\circ_2z)+y\circ_2(x\circ_1z)=0,
	\end{equation*}
\end{defn}
for all $x,y,z\in L.$ 

Moving forward, we revisit the definition of zinbiel algebras, which are intricately linked to Leibniz algebras in the framework described by 'J-L. Loday' in his work \cite{loday1995cup}. 

\begin{defn}
	A zinbiel algebra is a pair $(L,\ast)$ consisting of a vector space $L$ equipped with a linear map $\ast:L\otimes L \rightarrow L$ such that 
	\begin{equation*}
		x \ast (y\ast z)=(x\ast y) \ast z+(y\ast x)\ast z, \quad \forall~ x,y,z \in L.
	\end{equation*}
\end{defn}

Furthermore, 'J.-L. Loday' introduced the concept of diassociative algebras in his work \cite{loday2001dialgebras}. It's essential to note that the Koszul dual of the operad associated with diassociative algebras corresponds to the operad of dendriform algebras. 

\begin{defn}
	A dendriform algebra is a triple $(L,\prec,\succ)$ where $L$ is a vector space, bilinear maps $\prec,\succ : L \otimes L \rightarrow L$ such that for all $x,y,z \in L$, we have the following identities 
	\begin{align*}
		(x\prec y) \prec z&=x\prec(y\prec z+y\succ z),\\
		(x\succ y)\prec z&=x\succ(y\prec z), \\
		x\succ(y\succ z)&=(x\prec y+x\succ y)\succ z.
	\end{align*}
\end{defn}

Dendriform algebras find applications across diverse mathematical domains, including algebraic topology and homological algebra. They necessitate dendrification, which is a process that establishes connections between different algebraic structures exhibiting specific properties. One such property involves the incorporation of Rota-Baxter operators or endomorphism operators, along with the commutator bracket, as mentioned in the preceding section. For a more comprehensive understanding, please refer to the following references: \cite{aguiar2000pre, loday2001dialgebras, loday2003algebres}.

\begin{defn} \label{compatible dendri}
	A compatible dendriform algebra is quintuple $(L,\prec_1,\succ_1,\prec_2,\succ_2)$ consisting of a vector space $L$ together with four bilinear operations $\prec_1,\succ_1,\prec_2,\succ_2:L \otimes L \rightarrow L$ where $(L,\prec_1,\succ_1)$ and $(L,\prec_2,\succ_2)$ are both two dendriform algebras such that for $x,y,z \in L$
	\begin{eqnarray*}
		(x \prec_1 y)\prec_2 z+(x\prec_2y)\prec_1z&=&x\prec_2(y\prec_1z+y\succ_1z)+x\prec_1(y\prec_2z+y\succ_2z), \\
		(x\succ_1y)\prec_2z+(x\succ_2y)\prec_1z&=&x\succ_2(y\prec_1z)+x\succ_1(y\prec_2z), \\
		(x\prec_1y+x\succ_1y)\succ_2z+(x\prec_2y+x\succ_2y)\succ_1z&=&x\succ_2(y\succ_1z)+x\succ_1(y\succ_2z).
	\end{eqnarray*}
\end{defn}

\begin{re} \label{rem2.1}
	Let be $(L,\prec,\succ)$ be a dendriform algebra, then an associative algebra rise by defining the structure $\mu: L\otimes L \longrightarrow L $ such that  $\mu(x,y)=x\succ y + x\prec y.$\\ 
	Also  a pre-Lie algebra rise by defining the  structure  $\circ : L\otimes L  \longrightarrow L$ such that  
	$ x\circ y=x \succ y - y\prec x $.\\    
\end{re}



The following diagram explains the relationship among the above-mentioned algebraic structures.

\begin{align*}
	\xymatrix{
		\text{Zinbiel algebra} \ar@{->}[rr]^{\small{x\ast y=x \prec y=x\succ y}~~~~} & &  \text{Dendriform algebras}  \ar@{->}[d]_{\small{\mu(x,y)=x\succ y+x\prec y}}\ar@{->}[rr]^{\small{x \circ y=x\succ y-y \prec y}} & &\text{Pre-Lie algebras}\ar@<-2pt>[d]_{~~\small{[x,y]_c=x \circ y-y \circ x}}\\
		& & \text{Associative algebra}\ar@{->}[rr]^{~~~~~~\small{[x,y]_c=\mu(x,y)-\mu(y,x)}} & & \text{Lie algebra}\ar@<-2pt>[u]_{\small{x\circ y=[R(x),y]}~~~~}.
	}
\end{align*}



\section{Compatible associative and nonassociative algebras with derivation} \label{section 3}
In this section, we extend the concept of compatible algebras to include compatible Der pair structures. This extension involves the combination of two Der pairs, which are two algebras equipped with two derivations, in order to create a new Der pair with identical structural properties. We start this exploration by delving into \textbf{compatible AssDer pairs}.
\subsection{Compatible AssDer pairs} \label{subsection 3.1}
Inspired by the research conducted by Odesskii Alexander and Vladimir Sokolov on the examination of compatible associative algebras as refereed in \cite{odesskii2006compatible, odesskii2006algebraic, odesskii2006integrable}, we introduce the concept of a "compatible AssDer pair." This constitutes a compatible associative algebra that is equipped with a compatible derivation.
\begin{defn} \label{compa AssDer pairs}
	A \textbf{compatible AssDer pair} is a quintuple  $(A,\mu_{1},\mu_{2},\delta_{1},\delta_{2})$ where  $L$ is a vector space, $\mu_{1}$ and $\mu_{2}$ are associative algebra structures on $A$ with two derivations $\delta_{1}$ and $\delta_{2}$ such that, for all $x,y,z \in A$: 
	\begin{eqnarray}
		(L,\mu_1,\mu_2) &&\text{ is a compatible associative algebra } \\
		\delta_{1}(\mu_{2}(x,y))+\delta_2(\mu_1(x,y))&=&\mu_2(\delta_1x,y)+\mu_2(x,\delta_1y)+\mu_1(\delta_2x,y)+\mu_1(x,\delta_2y). \label{definition CADR assertion 2}
	\end{eqnarray}
	
\end{defn}
\begin{ex}
	Let $(A,\mu_1)$ and $(A,\mu_2)$ be two associative algebras such that $f_1$ and $f_2$ are $1$-cocycles, respectively, on $(A,\mu_1)$ and $(A,\mu_2)$ with coefficient in $A$. Then $(A,\mu_{1},\mu_{2},f_1,f_2)$ is a \textbf{compatible AssDer pair}.
\end{ex}
\begin{re}
	In an alternative manner, the previous definition implies that the combined sums $\mu_1+\mu_2$ and $\delta_1+\delta_2$ establish an AssDer pair, where $\delta_1+\delta_2$ represents a derivation acting on the multiplication operation $\mu_1+\mu_2$. We can also express that the two derivations, $\delta_1$ and $\delta_2$ are compatible, denoted simply as $(\delta_1, \delta_2)$. Therefore, a "compatible AssDer pair" can be defined as a set of compatible associative algebras $(A, \mu_1, \mu_2)$ equipped with a compatible derivation $(\delta_1, \delta_2)$.
\end{re}
\begin{defn}
	Let $(A_1,\mu_1,\mu_2,\delta_1,\delta_2)$ and $(A_2,\mu_1^{\prime},\mu_2^{\prime},\delta_1^{\prime},\delta_2^{\prime})$ be two \textbf{compatible AssDer pairs}. Then the linear map $\varphi:A_1 \rightarrow A_2$ is called a homomorphism if it is a homomorphism of associative algebras from $(A_1,\mu_1)$ (respectively $(A_1,\mu_2)$) to $(A_2,\mu_1^{\prime})$ (respectively $(A_2,\mu_2^{\prime})$) and the following identities are satisfied
	\begin{equation} \label{morphism compatible AssDer pair}
		\varphi \circ \delta_1=\delta_1^{\prime} \circ \varphi,\quad \varphi \circ \delta_2=\delta_2^{\prime} \circ \varphi.
	\end{equation}
\end{defn}
\begin{defn} \label{RBO on compatible AssDer pair}
	The linear map $R:A \rightarrow A$ is a Rota-Baxter operator on the \textbf{compatible AssDer pair} $(A,\mu_1,\mu_2,\delta_1,\delta_2)$ if it is a Rota-Baxter operator on $(A,\mu_1)$ (respectively $(A,\mu_2)$) and the following identities are satisfied 
	\begin{equation} \label{RBO on CADP}
		R \circ \delta_1=\delta_1 \circ R,\quad R \circ \delta_2=\delta_2 \circ R.
	\end{equation}
\end{defn}
Recall that a Rota-Baxter operator gives rise to an associative algebra denoted as $(A, \mu_{\mathrm{R}})$. Further elaboration on this topic can be found in \cite{chtioui2021co}. In this context, the multiplication operation is defined as follows, considering elements $x$ and $y$ in $A$:
\begin{equation*}
	\mu_{\mathrm{R}}(x,y)=\mu(Rx,y)+\mu(x,Ry).
\end{equation*}
With $(A,\mu)$ is an associative algebra. \\
In the next proposition we extend that result to the case of \textbf{compatible AssDer Pair}.
\begin{prop}
	Let $(A,\mu_1,\mu_2,\delta_1,\delta_2)$ be a \textbf{compatible AssDer pair} and $R:A \rightarrow A$ be a Rota-Baxter operator on $A$. Then $(A,\mu_{\mathrm{R}}^1,\mu_{\mathrm{R}}^2,\delta_1,\delta_2)$ is a \textbf{compatible AssDer pair} where, for $x,y \in A$
	\begin{equation*}
		\mu_{\mathrm{R}}^i(x,y)=\mu_{\mathrm{i}}(Rx,y)+\mu_{\mathrm{i}}(x,Ry),\quad i\in \{1,2\}.
	\end{equation*}
\end{prop}
\begin{proof}
	For $x,y,z \in A$ we have 
	\begin{align*}
		\mu_{\mathrm{R}}^2(x,\mu_{\mathrm{R}}^1(y,z))&=\mu_2(Rx,\mu_{\mathrm{R}}^1(y,z))+\mu_2(x,R(\mu_{\mathrm{R}}^1(y,z))) \\
		&=\mu_2(Rx,\mu_1(Ry,z))+\mu_2(Rx,\mu_1(y,Rz))+\mu_2(x,R\big(\mu_1(Ry,z)+\mu_1(y,Rz) \big)) \\
		&=\mu_2(Rx,\mu_1(Ry,z))+\mu_2(Rx,\mu_1(y,Rz))+\mu_2(x,\mu_1(Ry,Rz)) \quad (\text{ because } R \text{ is a Rota-Baxter operator})
	\end{align*}
	So similarly and with simple calcul we obatin 
	\begin{eqnarray*}
		\mu_{\mathrm{R}}^2(x,\mu_{\mathrm{R}}^1(y,z))&=&\mu_2(Rx,\mu_1(Ry,z))+\mu_2(Rx,\mu_1(y,Rz))+\mu_2(x,\mu_1(Ry,Rz)) \\
		\mu_{\mathrm{R}}^1(x,\mu_{\mathrm{R}}^2(y,z))&=&\mu_1(Rx,\mu_2(Ry,z))+\mu_1(Rx,\mu_2(y,Rz))+\mu_1(x,\mu_2(Ry,Rz)) \\
		\mu_{\mathrm{R}}^2(\mu_{\mathrm{R}}^1(x,y),z)&=&\mu_2(\mu_1(Rx,Ry),z)+\mu_2(\mu_1(Rx,y),Rz)+\mu_2(\mu_1(x,Ry),Rz) \\
		\mu_{\mathrm{R}}^1(\mu_{\mathrm{R}}^2(x,y),z)&=&\mu_1(\mu_2(Rx,Ry),z)+\mu_1(\mu_2(Rx,y),Rz)+\mu_1(\mu_2(x,Ry),Rz)
	\end{eqnarray*}
	By using equation \eqref{comp associative algebras} we obatin that 
	\begin{equation*}
		\mu_{\mathrm{R}}^2(x,\mu_{\mathrm{R}}^1(y,z))+\mu_{\mathrm{R}}^1(x,\mu_{\mathrm{R}}^2(y,z))=\mu_{\mathrm{R}}^2(\mu_{\mathrm{R}}^1(x,y),z)+\mu_{\mathrm{R}}^1(\mu_{\mathrm{R}}^2(x,y),z).
	\end{equation*}
	Also we have 
	\begin{align*}
		\delta_1 \circ \mu_{\mathrm{R}}^1(x,y)&=\delta_1 (\mu_1(Rx,y)+\mu_1(x,Ry)) \\
		&=\mu_1(\delta_1\circ Rx,y)+\mu_1(Rx,\delta_1y)+\mu_1(\delta_1x,Ry)+\mu_1(x,\delta_1\circ Ry) \\
		&\overset{\eqref{RBO on CADP}}{=}\mu_1(R\circ \delta_1x,y)+\mu_1(\delta_1x,Ry)+\mu_1(Rx,\delta_1y)+\mu_1(x,R\circ \delta_1y) \\
		&=\mu_{\mathrm{R}}^1(\delta_1x,y)+\mu_{\mathrm{R}}^1(x,\delta_1y)
	\end{align*}
	which mean $\delta_1$ is a derivation on $(A,\mu_{\mathrm{R}}^1)$ then $(A,\mu_{\mathrm{R}}^1,\delta_1)$ is an AssDer pair. \\
	Similarly we prove that $(A,\mu_{\mathrm{R}}^2,\delta_2)$ is an AssDer pair. \\
	Finaly, by using the fact that $(A,\mu_1,\mu_2,\delta_1,\delta_2)$ we have 
	\begin{align*}
		\delta_1(\mu_R^2(x,y))+\delta_2(\mu_R^1(x,y))&=\delta_1(\mu_2(Rx,y)+\mu_2(x,Ry))+\delta_2(\mu_1(Rx,y)+\mu_1(x,Ry)) \\
		&\overset{\eqref{RBO on CADP}}{=}\mu_2(R\circ \delta_1x,y)+\mu_2(\delta_1x,Ry)+\mu_2(Rx,\delta_1y)+\mu_2(x,\delta_1\circ Ry) \\
		&+\mu_1(R\circ \delta_2x,y)+\mu_1(\delta_2x,Ry)+\mu_1(Rx,\delta_2y)+\mu_1(x,\delta_2\circ Ry) \\
		&=\mu_R^2(\delta_1x,y)+\mu_R^2(x,\delta_1y)+\mu_R^1(\delta_2x,y)+\mu_R^1(x,\delta_2y).
	\end{align*}
	This complete the proof.
\end{proof}
\subsection{Compatible LieDer pairs}\label{sec3}
In this subsection, we present the concept of \textbf{compatible LieDer pairs}, an extension of the LieDer pair concept, which is further elaborated upon in \cite{tang2019cohomologies}.

Consider two LieDer pairs, denoted as $(L, [\cdot,\cdot]_{1}, \delta_{1})$ and $(L, [\cdot,\cdot]_{2}, \delta_{2}).$

\begin{defn} \label{def2.1}
	A \textbf{compatible  LieDer pair} is quintuple  $(L,[ \cdot,\cdot]_{1},[\cdot,\cdot]_{2},\delta_{1},\delta_{2})$ where $L$ is a vector space, $[\cdot,\cdot]_{1}$ and $[\cdot,\cdot]_{2}$ are two Lie algebras structure on $L$ equipped respectively  with two derivations  $\delta_{1}$ and $\delta_{2}$ such that $\space \forall x,y,z\in L$ :
	\begin{eqnarray}
		\circlearrowleft_{x,y,z}[x,[y,z]_{1}]_{2}&+&\circlearrowleft_{x,y,z}[x,[y,z]_{2}]_{1}=0, \label{Def CLDP 1} \\
		\delta_{1}[x,y]_2+\delta_2[x,y]_1 &=& [\delta_1x,y]_2+[x,\delta_1y]_2+[\delta_2x,y]_1+[x,\delta_2]_1. \label{Def CLDP 2}
	\end{eqnarray}		
\end{defn}
\begin{ex}
	Let $(L,[\cdot,\cdot]_1)$ and $(L,[\cdot,\cdot]_2)$ be two Lie algebras such that $f_1$ and $f_2$ are $1$-cocycles, respectively, on $(L,[\cdot,\cdot]_1)$ and $(L,[\cdot,\cdot]_2)$ with respect to the adjoint representation. Then $(L,[\cdot,\cdot]_1,[\cdot,\cdot]_2,f_1,f_2)$ is a \textbf{compatible LieDer pair}.	
\end{ex}
\begin{re}
	The previous definition can be reformulated as follows: Suppose we have two "LieDer pairs," denoted as $(L,[\cdot,\cdot]_1,\delta_1)$ and $(L,[\cdot,\cdot]_2,\delta_2)$. Then, when we combine these two Der pairs into a single Der pair $(L,[\cdot,\cdot]_1,[\cdot,\cdot]_2,\delta_1,\delta_2)$, we obtain a \textbf{compatible LieDer pair}. This combination satisfies the two provided identities, \eqref{Def CLDP 1} and \eqref{Def CLDP 2} which signify the compatibility of the two brackets and derivations. Equivalently, this can be expressed as $(L,[\cdot,\cdot]_1+[\cdot,\cdot]_2,\delta_1+\delta_2)$ forming a LieDer pair.
\end{re}
\begin{prop}
	A  quintuple  $(L,[\cdot,\cdot]_{1},[\cdot,\cdot]_{2},\delta_{1},\delta_{2})$ is a \textbf{compatible LieDer pair} if and only if  $[\cdot,\cdot]_{1} $ and $[\cdot,\cdot]_{2}$ are two lie  algebra structures on $L$ such that for any $k_{1}, k_{2},p_{1},p_{2} \in \mathbb{K}$ and $x,y \in L$, the following structure:
	\begin{align*}
		[\![ x,y]\!]&=k_{1} [ x,y]_1+k_{2} [ x,y]_2, \\
		\delta &=p_{1}\delta_{1}+p_{2}\delta_{2}.
	\end{align*}
	defines a LieDer pair structure on $L$. 
\end{prop}

\begin{proof}
	For the first sense, suppose that $(L,[\cdot,\cdot]_{1},[\cdot,\cdot]_{2},\delta_{1},\delta_{2})$ is a \textbf{compatible LieDer pair}, we need just to prove that $\delta$ is a derivation on $(L,[\![\cdot,\cdot]\!])$. \\
	Let be $x,y\in L$.
	\begin{eqnarray*}
		\delta([\![x,y]\!])&=&(p_{1}\delta_{1}+p_{2}\delta_{2}) \bigg(k_{1}[x,y]_{1}+k_{2}[x,y]_{2}\bigg)\\
		&=& p_{1}k_{1}\delta_{1}[x,y]_{1}+p_{2}k_{2}\delta_{2}[x,y]_{2}+p_{1}k_{2}\delta_{1}[x,y]_{2}+p_{2}k_{1}\delta_{2}[x,y]_{1}\\
		&=& p_{1}k_{1}[\delta_{1}x,y]_{1}+p_{1}k_{1}[x,\delta_{1}y]_{1}+p_{2}k_{2}[\delta_{2}x,y]_{2}+p_{2}k_{2}[x,\delta_{2}y]_{2}+p_{1}k_{2}[\delta_{1}x,y]_{2}\\
		&+& p_{1}k_{2}[x,\delta_{1}y]_{2}+p_{2}k_{1}[\delta_{2}x,y]_{1}
		+p_{2}k_{1}[x,\delta_{2}y]_{1}\\
		&=& k_{1}[p_{1}\delta_{1}x,y]_{1}+k_{1}[x,p_{1}\delta_{1}y]_{1}+k_{2}[p_{2}\delta_{2}x,y]_{2}+k_{2}[x,p_{2}\delta_{2}y]_{2}+k_{2}[p_{1}\delta_{1}x,y]_{2}\\
		&+& k_{2}[x,p_{1}\delta_{1}y]_{2}+k_{1}[p_{2}\delta_{2}x,y]_{1}
		+k_{1}[x,p_{2}\delta_{2}y]_{1}\\
		&=& k_{1}[p_{1}\delta_{1}x,y]_{1}+k_{2}[p_{2}\delta_{2}x,y]_{2}+k_{1}[x,p_{1}\delta_{1}y]_{1}+k_{2}[x,p_{2}\delta_{2}y]_{2}+k_{1}[p_{2}\delta_{2}x,y]_{1}\\&+&  k_{2}[p_{1}\delta_{1}x,y]_{2}
		+k_{1}[x,p_{2}\delta_{2}y]_{1}+ k_{2}[x,p_{1}\delta_{1}y]_{2}\\
		&=&[\![p_{1}\delta_{1}x,y]\!]+[\![p_{2}\delta_{2}x,y]\!]+[\![x,p_{1}\delta_{1}y]\!]+[\![x,p_{2}\delta_{2}y]\!]\\
		&=& [\![(p_{1}\delta_{1}+p_{2}\delta_{2})x,y]]+[\![x,(p_{1}\delta_{1}+p_{2}\delta_{2})y]\!]\\
		&=& [\![\delta x,y]\!]+[\![x,\delta y ]\!].
	\end{eqnarray*}
	
	Which mean $(L,[\![\cdot,\cdot]\!],\delta)$ is a LieDer pair.\\
	For  the opposite sense, we need just to prove that,
	$$\delta_{1}[x,y]_2+\delta_2[x,y]_1 = [\delta_1x,y]_2+[x,\delta_1y]_2+[\delta_2x,y]_1+[x,\delta_2]_1$$
	In the first hand we have:
	$$\delta[\![x,y]\!]=p_{1}k_{1}\delta_{1}[x,y]_{1}+p_{2}k_{2}\delta_{2}[x,y]_{2}+p_{1}k_{2}\delta_{1}[x,y]_{2}+p_{2}k_{1}\delta_{2}[x,y]_{1}.$$
	And in the second hond we have 
	$$[\![\delta x,y]\!]+[\![x,\delta y]\!]= p_{1}k_{1}\delta_{1}[x,y]_{1}+p_{2}k_{2}\delta_{2}[x,y]_{2}+p_{1}k_{2}\left( [\delta_{1} x,y]_{2}+[x,\delta_{1}y]_{1}\right) +p_{2}k_{1}\left( [\delta_{2} x,y]_{1}+[x,\delta_{2}y]_{1}\right).$$
	
	By identification we obtain the result.

\end{proof}

\begin{defn}
	A homomorphism between two \textbf{compatible LieDer pairs} $(L_1, [\cdot,\cdot]_1,[\cdot,\cdot]_2,\delta_1,\delta_2)$ and $(L_2, [\cdot,\cdot]_1^{\prime},[\cdot,\cdot]_2^{\prime},\delta_1^{\prime},\delta_2^{\prime})$ is a linear map $\varphi :L_1\longrightarrow L_2$ such that $\varphi$ is a LieDer pair homomorphism between $(L_1, [\cdot,\cdot]_1,\delta_1)$ (respectively $(L_1, [\cdot,\cdot]_2,\delta_2)$) and $(L_2,[\cdot,\cdot]_1^{\prime},\delta_1^{\prime})$ (respectively $(L_2, [\cdot,\cdot]_2^{\prime},\delta_2^{\prime})$).
\end{defn}
\begin{re}
	Recall that a LieDer pair homomorphism means that, for $i \in \{1,2\}$:
	\begin{eqnarray*}
		\varphi([x,y]_{\mathrm{i}})&=&[\varphi(x),\varphi(y)]_{\mathrm{i}}^{\prime}, \\
		\varphi \circ \delta_{\mathrm{i}}&=&\delta_{\mathrm{i}}^{\prime} \circ \varphi.
	\end{eqnarray*}
\end{re}
\begin{defn} 
	The linear map $R:L \rightarrow L$ is a Rota-Baxter operator on the \textbf{compatible LieDer pair} $(L, [\cdot,\cdot]_1,[\cdot,\cdot]_2,\delta_1,\delta_2)$ if it is a Rota-Baxter operator on $(L,[\cdot,\cdot]_1)$ (respectively $(L,[\cdot,\cdot]_2)$ ) and the following identities are satisfied 
	\begin{equation*} 
		R \circ \delta_1=\delta_1 \circ R,\quad R \circ \delta_2=\delta_2 \circ R.
	\end{equation*}
\end{defn}
\subsection{Compatible DenriDer and zinDer pairs }\label{subsection 3.3}
In this subsection, we expand upon the concept of compatible dendriform algebras to encompass compatible Dendriform algebras with derivations, referred to as \textbf{compatible DendriDer pairs}. However, before proceeding with this, we shall first define the concept of a dendriDer pair. Subsequently, in the second step, we will establish the definitions of compatible zinbiel algebras, zinDer pairs (zinbiel algebras with derivations), and \textbf{compatible zinDer pairs}.
\paragraph*{Compatible DendriDer pairs} 
In this subsection we define the DendriDer pairs then we extend this notion to the case of \textbf{compatible dendriDer pairs}.
\begin{defn} \label{def4.1}
	A dendriDer pair is a quadruple $(L,\prec,\succ,\delta)$ where $(L,\prec,\succ)$ is a dendriform algebra and $\delta:L \longrightarrow L $ is a derivation on $L$ .i.e, for all $x,y\in L$:\\
	$$  \delta(x\prec y )=\delta x \prec y + x\prec \delta y,$$
	$$\delta(x\succ y )=\delta x \succ y + x\succ\delta y.$$
\end{defn}
Because Rota-Baxter operators are connected with dendriform algebras in an active area of research we recall its definition as follow
\begin{defn}
	Let $(L,\prec,\succ)$ be a dendriform algebra and $R:L \rightarrow L$ be a linear map. Then $R$ is said to be a Rota-Baxter operator on $(L,\prec,\succ)$ if is satisfy for $x,y \in L$
	\begin{eqnarray*}
		Rx\prec Ry&=&R(Rx\prec y+x\prec Ry), \\
		Rx\succ Ry&=&R(Rx\succ y+x\succ Ry). 
	\end{eqnarray*}
\end{defn}	
Here we cite an example of dendriDer pair.
\begin{ex}
	Let $R$ be an invertible Rota-Baxter operator on $(L,\prec,\succ)$. Then $(L,\prec,\succ,R^{-1})$ is a dendriDer pair.
\end{ex}
\begin{defn}
	A  homomorphism between two  dendriDer pairs $(L_1,\prec_{1},\succ_{1},\delta_{1})$ and $(L_2,\prec_2,\succ_{2},\delta_{2})$ is a linear map $\varphi: L_1\longrightarrow L_2 $ such that $\varphi $ is homomorphism between $(L_1,\prec_{1},\succ_{1})$ and $(L_2,\prec_{2},\succ_{2})$ and the following condition hold: 
	\begin{equation*}
		\varphi \circ \delta_{1}=\delta_{2}\circ \varphi.
	\end{equation*}
\end{defn}
\begin{defn}\label{comp DendriDer pair}
	A \textbf{compatible dendriDer pair} is a $7$-uple $(L,\prec_{1},\succ_{1},\prec_{2},\succ_{2},\delta_{1},\delta_{2})$ consisting of vector space $L$ equipped with four bilinear maps $\prec_{1},\succ_{1},\prec_{2},\succ_{2}:L\otimes L \longrightarrow L$ and two derivations $\delta_1,\delta_2$ such that $(L,\prec_{1},\succ_{1},\delta_{1})$ and $(L,\prec_{2},\succ_{2},\delta_{2})$ are two dendriDer pairs. Where the following conditions are satisfied:
	\begin{eqnarray}
		(L,\prec_{1},\succ_{1},\prec_{2},\succ_{2})&& \text{ is a compatible dendriform algebra (see definition } \eqref{compatible dendri}) , \\
		\delta_{1}(x\prec_{2}y)+\delta_{2}(x\prec_{1}y)&=&\delta_{1}x\prec_{2}y+x\prec_{2}\delta_{1}y+\delta_{2}x\prec_{1}y+x\prec_{1}\delta_{2}y, \label{comp dendrDer 2}\\
		\delta_{1}(x\succ_{2}y)+\delta_{2}(x\succ_{1}y)&=&\delta_{1}x\succ_{2}y+x\succ_{2}\delta_{1}y+\delta_{2}x\succ_{1}y+x\succ_{1}\delta_{2}y.\label{comp dendrDer 3}
	\end{eqnarray}
\end{defn}
\begin{re}
	The previous definition can be seen as $(L,\prec_1+\prec_2,\succ_1+\succ_2,\delta_1+\delta_2)$ is a dendriDer pair which mean that the two following dendriDer pairs $(L,\prec_1,\succ_1,\delta_1)$ and $(L,\prec_2,\succ_2,\delta_2)$ are compatible.
\end{re}
\paragraph*{Compatible zinDer pairs}
In this subsection, we begin by introducing zinDer pairs, which are zinbiel algebras with derivations. Following that, we delve into the concept of compatible zinbiel algebras. Finally, we proceed to generalize these ideas to define \textbf{compatible zinDer pairs}.

To begin, let's first revisit the definition of a derivation on a zinbiel algebra.
\begin{defn} \cite{almutairi2018derivations}
	Let $(L,\ast)$ be a zinbiel algebra, a derivation on $(L,\ast)$ is a linear map $\delta:L \rightarrow L$ satisfying, for $x,y \in L$
	\begin{equation*} 
		\delta(x \ast y)=\delta x\ast y+x \ast \delta y.
	\end{equation*}
\end{defn}
\begin{defn} \label{ZinDer pairs definition}
	A zinDer pair is a triple $(L,\ast,\delta)$ consisting of a zinbiel algebra $(L,\ast)$ and a derivation $\delta$.
\end{defn}
\begin{ex}
	Let $L$ be a two dimensional zinbiel algebra with composition law $e_1\ast e_1=e_2$ on the basis $\{e_1,e_2\}$. Then $(L,\ast,\delta)$ is a zinDer pair with, $\alpha,\beta \in \mathbb{K}$, 
	$$ \begin{pmatrix}
		\alpha & 0 \\
		\beta & 2\alpha 
	\end{pmatrix}
	$$
	For more details see \cite{almutairi2018derivations}.
\end{ex}
\begin{defn}
	Let $(L_1,\ast_1,\delta_1)$ and $(L_2,\ast_2,\delta_2)$ be two zinDer pairs. A zinDer pair homomorphism $\varphi:L \rightarrow L$ is a linear map satisfying, for $x,y \in L$
	\begin{eqnarray*}
		\varphi(x\ast_1y)&=&\varphi(x) \ast_2 \varphi(y), \\
		\varphi \circ \delta_1&=&\delta_2 \circ \varphi.
	\end{eqnarray*}  
\end{defn}
Now, we proceed to the introduction of the concept of compatible zinbiel algebras, which entails the combination of two zinbiel algebras, denoted as $(L,\ast_1)$ and $(L,\ast_2),$ to form a new algebra within the same algebraic structure.
\begin{defn} \label{compatible zinbiel algebras}
	A compatible zinbiel algebra is a triple $(L,\ast_1,\ast_2)$ consisting of a vector space endowed with two bilinear maps $\ast_1,\ast_2:L \rightarrow L$ such that $(L,\ast_1)$ and $(L,\ast_2)$ are both zinbiel algebras such that for $x,y,z \in L$
	\begin{equation}\label{comp zinbiel}
		x\ast_1(y\ast_2z)+x\ast_2(y\ast_1z)=(x\ast_2y)\ast_1z+(x\ast_1y)\ast_2z+(y\ast_2x)\ast_1z+(y\ast_1x)\ast_2z.
	\end{equation}
\end{defn}
Using the previous result in this subsection we introduce the \textbf{compatible zinDer pairs}.
\begin{defn} \label{compatible zinDer pairs def}
	A \textbf{compatible zinDer pair} is quintuple $(L,\ast_1,\ast_2,\delta_1,\delta_2)$ consisting of a vector space $L$ equipped with two bilinear maps $\ast_1,\ast_2:L\otimes L \rightarrow L$ and a two derivations $\delta_1,\delta_2$ such that $(L,\ast_1,\delta_1)$ and $(L,\ast_2,\delta_2)$ are two zinDer pairs where the following conditions holds, for $x,y \in L$
	\begin{eqnarray}
		(L,\ast_1,\ast_2) &&\text{ is a compatible zinbiel algebra (see } \eqref{comp zinbiel})  ,\\
		\delta_1(x\ast_2y)+\delta_2(x\ast_1y)&=&\delta_1x\ast_2y+x\ast_2\delta_1y+\delta_2x\ast_1y+x\ast_1\delta_2y. \label{compatibility derivation zinbiel}
	\end{eqnarray}
\end{defn}
We summariz our results of the last subsection in the following diagramm:
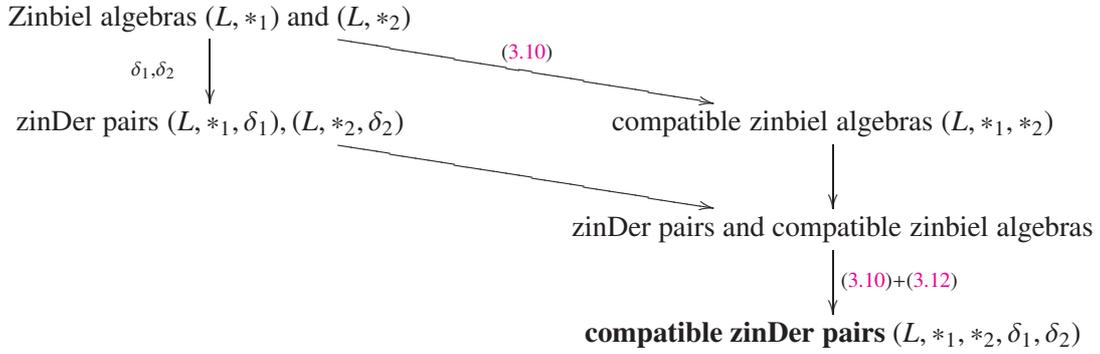
\begin{align*}
	\xymatrix{
		\text{Zinbiel algebras } (L,\ast_1) \text{ and } (L,\ast_2)  \ar@{->}[d]_{\small{\delta_1,\delta_2}~~~~}\ar@{->}[drr]^{\small{\eqref{comp zinbiel}}} \\
		\text{zinDer pairs }(L,\ast_1,\delta_1), (L,\ast_2,\delta_2)\ar@{->}[drr] & &  \text{compatible zinbiel algebras }(L,\ast_1,\ast_2)\ar@{->}[d] \\
		& &  \text{zinDer pairs and compatible zinbiel algebras} \ar@{->}[d]^{\small{\eqref{comp zinbiel}}+\eqref{compatibility derivation zinbiel}}  \\
		&& \textbf{compatible zinDer pairs }(L,\ast_1,\ast_2,\delta_1,\delta_2)
	}
\end{align*}

\section{Dendrification of  some compatible algebras with derivation} \label{section 4}
This section is devoted to examining the interconnections between various algebraic structures, a concept called "dendrification." This concept is distinguished by several properties, which we will outline in the following.
\subsection{Dendrification of some algebras with derivation} \label{subsection 4.1}
In this subsection, we investigate the relationships between distinct algebras equipped with derivations, often referred to as Der pair structures.
\begin{prop}\label{prop dendriDer to assDer}
	Let be $(L,\prec,\succ,\delta)$ be a dendriDer pair. Define the bilineair map 
	\begin{equation}
		\mu :  L\otimes L \longrightarrow L   \text{   by  } \ \   \mu (x,y)=x\succ y +x\prec y ,\ \ \  \forall  x,y\in L.
	\end{equation}
	\\  Then $(L, \mu, \delta )$ is an AssDer pair.
\end{prop}
\begin{proof}
	In \cite{aguiar2000pre}, the previous proposition holds in the case of algebras. So we need just to prove that $\delta $ is a derivation on $(L, \mu)$.\\
	Let be $x,y\in  L$
	\begin{eqnarray*}
		\delta(\mu(x,y))&=& \delta (x\succ y + x\prec y )\\
		&=& \delta (x\succ y) +\delta( x\prec y )\\
		&=& \delta x \succ y + x \succ \delta y  + \delta x \prec y + x\prec \delta y  \\
		&=&  ( \delta x \succ y  + \delta x \prec y ) + (x \succ \delta y + x\prec \delta y )\\
		&=& \mu (\delta x , y) + \mu(x,\delta y ).
	\end{eqnarray*}
	This complete the proof.
\end{proof}
\begin{center}
	$\xymatrix{
		\emph{dendriDer pair}&\underrightarrow{ \mu(x,y)=x\prec y+x \succ y} &\emph{Ass Der pair}
	}$
\end{center}
\begin{prop}
	Let be  $( L, \prec, \succ , \delta )$ be a  dendriDer pair. Define 
	$$\circ : L\otimes L \longrightarrow L  \text{ by } x\circ y = x\succ y - y\prec x ; \ \ \ \forall x,y\in L,$$  then $(L , \circ , \delta )$ is a pre-LieDer pair.
\end{prop}
\begin{center}
	$\xymatrix{
		\emph{dendriform Der pair}&\underrightarrow{ x\circ y=x\prec y-y \succ x} &\emph{pre-LieDer pair}
	}$
\end{center}

In the next proposition we define a commutative Ass-Der pair $(L,\mu,\delta)$ from a zinbiel Der pair $(L,\ast,\delta)$.
\begin{prop} \label{prop zinDer to AssDer pair}
	Let be $(L,\ast,\delta)$  a zinbiel Der pair, define $\mu(x,y)=x\ast y + y\ast x $. Then $(L,\mu,\delta)$ is an Ass-Der pair. 
\end{prop}
\begin{proof}
	In \cite{aguiar2000pre}, when $\alpha=Id_{L}$. The previous proposition holds in case of algebras. So we need just to prove that $\delta$ is derivation on $(L, \mu)$.\\
	Let be $x,y\in L$
	\begin{eqnarray*}
		\delta(\mu(x,y))&=&  \delta(x\ast y+ y\ast x)\\
		&=& \delta(x\ast y)+\delta(y\ast x)\\
		&=& \delta x \ast y + x \ast \delta y  + \delta y \ast x + y \ast \delta x \\
		&=& (\delta x \ast y + y\ast \delta x)+( x\ast \delta y + \delta y \ast x)\\
		&=& \mu(\delta x , y)+ \mu(x, \delta y ).
	\end{eqnarray*}
	This complete the proof.
	
\end{proof}
\begin{prop}
	Let $( L,\ast, \delta)$ be a zinbiel Der  pair, if $x\ast y =x\succ y = x\prec y.$ The $(L,\prec, \succ)$ is a dendriform Der pair. 
	
\end{prop}
We can summarize all the previous resuf in the following diagram.
\label{diagram1}
\begin{align*} 
	\xymatrix{
		\text{ZinDer pairs} \ar@{->}[rr]^{\small{x\ast y=x \prec y=x\succ y}~~~~} & &  \text{DendriDer pairs}  \ar@{->}[d]_{\small{\mu(x,y)=x\succ y+x\prec y}}\ar@{->}[rr]^{\small{x \circ y=x\succ y-y \prec y}} & &\text{Pre-LieDer pairs}\ar@<-2pt>[d]_{~~\small{[x,y]_c=x \circ y-y \circ x}}\\
		& & \text{AssDer pairs}\ar@{->}[rr]^{~\small{[x,y]_c=\mu(x,y)-\mu(y,x)}} & & \text{LieDer pairs}\ar@<-2pt>[u]_{\small{x\circ y=[R(x),y]}~~~~}.
	}
\end{align*}

\subsection{Dendrification of compatible ZinDer, AssDer, LieDer and pre-LieDer pairs} \label{subsection 4.2}
In this subsection, we extend the previous result to the case of compatible Der pair structures.
\begin{prop} \label{prop3.5}
	Let be $(L,\mu_{1},\mu_{2},\delta_{1},\delta_{2})$ a \textbf{compatible AssDer pair}, then $(L,[\cdot,\cdot]_{1},[\cdot,\cdot]_{2},\delta_{1},\delta_{2})$ is a \textbf{compatible LieDer pair} where, $x,y\in L$ 
	$$ [x,y]_{1}=\mu_{1}(x,y)-\mu_{1}(y,x),$$
	$$[x,y]_{2}=\mu_{2}(x,y)-\mu_{2}(y,x).$$
\end{prop}
\begin{proof}
	For $i\in\{1,2\}$ and $x,y,z\in L$ we have \\
	$[x,y]_{\mathrm{i}}=\mu_{\mathrm{i}}(x,y)-\mu_{\mathrm{i}}(y,x)=-(\mu_{\mathrm{i}}(y,x)- \mu_{\mathrm{i}}(x,y))=-[y,x]_{\mathrm{i}},$\\
	with simple calcul we prove that $\circlearrowleft_{x,y,z}[x,[y,z]_{\mathrm{i}}]_{\mathrm{i}}=0.$
	\begin{eqnarray*}
		\delta_{\mathrm{i}}([x,y]_{\mathrm{i}})&=& \delta_{\mathrm{i}}(\mu_{\mathrm{i}}(x,y)-\mu_{\mathrm{i}}(y,x))\\
		&=& \mu_{\mathrm{i}}(\delta_{\mathrm{i}}x,y)+\mu_{\mathrm{i}}(x,\delta_{\mathrm{i}}y)-\mu_{\mathrm{i}}(\delta_{\mathrm{i}}y,x)-\mu_{\mathrm{i}}(y,\delta_{\mathrm{i}}x)\\
		&=& [\delta_{\mathrm{i}}x,y]_{\mathrm{i}}	+[x,\delta_{\mathrm{i}}y]_{\mathrm{i}}.
	\end{eqnarray*}
	which mean $(L,[\cdot,\cdot]_{\mathrm{i}},\delta_{\mathrm{i}})$ is a Lie Der pair.\\
	Let's now prove that $(L,[\cdot,\cdot]_{1},[\cdot,\cdot]_{2},\delta_{1},\delta_{2})$ is a \textbf{compatible Lie Der pair}.
	We need first to prove that 
	$$ \circlearrowleft_{x,y,z}[[x,y]_{1},z]_{2}+\circlearrowleft_{x,y,z}[[x,y]_{2},z]_{1}=0.$$
	We have by using $[x,y]_i=\mu_i(x,y)-\mu_i(y,x)$
	\begin{eqnarray*}
		\circlearrowleft_{x,y,z}[[x,y]_{1},z]_{2}+\circlearrowleft_{x,y,z}[[x,y]_{2},z]_{1}&=& [[x,y]_{1}z]_{2}+ [[z,x]_{1}y]_{2}+[[y,z]_{1}x]_{2}\\
		&+&  [[x,y]_{2}z]_{1}+ [[z,x]_{2}y]_{1}+[[y,z]_{2}x]_{1},
	\end{eqnarray*}
	on the other hand we have,
	\begin{eqnarray*}
		[[x,y]_{1},z]_{2}&=& [\mu_{1}(x,y)-\mu_{1}(y,x),z]_{2}\\
		&=& [\mu_{1}(x,y),z]_{2}-[\mu_{1}(y,x),z]_{2}\\
		&=& \mu_{2}(\mu_{1}(x,y),z)-\mu_{2}(z,\mu_{1}(x,y))-\mu_{2}(\mu_{1}(y,x),z)+\mu_{2}(z,\mu_{1}(y,x)),
	\end{eqnarray*}
	$$[[z,x]_{1},y]_{2}= \mu_{2}(\mu_{1}(z,x),y)-\mu_{2}(y,\mu_{1}(z,x))-\mu_{2}(\mu_{1}(x,z),y)+\mu_{2}(y,\mu_{1}(x,z)),$$
	$$[[y,z]_{1},x]_{2}= \mu_{2}(\mu_{1}(y,z),x)-\mu_{2}(x,\mu_{1}(y,z))-\mu_{2}(\mu_{1}(z,y),x)+\mu_{2}(x,\mu_{1}(z,y)),$$
	$$ [[x,y]_{2},z]_{1}=\mu_{1}(\mu_{2}(x,y),z)-\mu_{1}(z,\mu_{2}(x,y))-\mu_{1}(\mu_{2}(y,x),z)+\mu_{1}(z,\mu_{2}(y,x)),$$
	$$[[z,x]_{2},y]_{1}= \mu_{1}(\mu_{2}(z,x),y)-\mu_{1}(y,\mu_{2}(z,x))-\mu_{1}(\mu_{2}(x,z),y)+\mu_{1}(y,\mu_{2}(x,z)),$$
	$$[[y,z]_{2},x]_{1}= \mu_{1}(\mu_{2}(y,z),x)-\mu_{1}(x,\mu_{2}(y,z))-\mu_{1}(\mu_{2}(z,y),x)+\mu_{1}(x,\mu_{2}(z,y)),$$
	by using equation \eqref{comp associative algebras} we obtaint he following result;
	\begin{eqnarray*}
		\circlearrowleft_{x,y,z}[[x,y]_{1},z]_{2}+\circlearrowleft_{x,y,z}[[x,y]_{2},z]_{1}&=& \mu_{2}(\mu_{1}(x,y),z)-\mu_{2}(z,\mu_{1}(x,y))-\mu_{2}(\mu_{1}(y,x),z)+\mu_{2}(z,\mu_{1}(y,x))\\
		&+& \mu_{2}(\mu_{1}(z,x),y)-\mu_{2}(y,\mu_{1}(z,x))-\mu_{2}(\mu_{1}(x,z),y)+\mu_{2}(y,\mu_{1}(x,z))\\
		&+& \mu_{2}(\mu_{1}(y,z),x)-\mu_{2}(x,\mu_{1}(y,z))-\mu_{2}(\mu_{1}(z,y),x)+\mu_{2}(x,\mu_{1}(z,y))\\
		&+& \mu_{1}(\mu_{2}(x,y),z)-\mu_{1}(z,\mu_{2}(x,y))-\mu_{1}(\mu_{2}(y,x),z)+\mu_{1}(z,\mu_{2}(y,x))\\
		&+& \mu_{1}(\mu_{2}(z,x),y)-\mu_{1}(y,\mu_{2}(z,x))-\mu_{1}(\mu_{2}(x,z),y)+\mu_{1}(y,\mu_{2}(x,z))\\
		&+& \mu_{1}(\mu_{2}(y,z),x)-\mu_{1}(x,\mu_{2}(y,z))-\mu_{1}(\mu_{2}(z,y),x)+\mu_{1}(x,\mu_{2}(z,y))\\
		&=& 0.
	\end{eqnarray*}
	the last step of the proof,
	\begin{align*}
		\delta_1[x,y]_2+\delta_2[x,y]_1&=\delta_1(\mu_2(x,y)-\mu_2(y,x)
		)+\delta_2(\mu_1(x,y)-\mu_1(y,x)) \\
		&=(\delta_1\circ\mu_2(x,y)+\delta_2\circ\mu_1(x,y))-(\delta_1\circ\mu_2(y,x)+\delta_2\circ \mu_1(y,x)) \\
		&\overset{\eqref{definition CADR assertion 2}}{=}\mu_2(\delta_1x,y)+\mu_2(x,\delta_1y)+\mu_1(\delta_2x,y)+\mu_1(x,\delta_2y)-\mu_2(\delta_1y,x)-\mu_2(y,\delta_1x)-\mu_1(\delta_2y,x)-\mu_1(y,\delta_2x) \\
		&=[\delta_1x,y]_2+[x,\delta_1y]_2+[\delta_2x,y]_1+[x,\delta_2y]_1
	\end{align*}
	So $(L,[\cdot,\cdot]_{1},[\cdot,\cdot]_{2},\delta_{1},\delta_{2})$ is a \textbf{compatible Lie Der pair}.
\end{proof}
So obtain the following diagram
\begin{center}
	$\xymatrix{
		\textbf{compatible AssDer pairs}&\underrightarrow{ commutator} &\textbf{\emph{compatible LieDer pairs}}
	}$
\end{center}
We will now demonstrate an alternative approach for constructing a  \textbf{ compatible LieDer pair } from a \textbf{ compatible AssDer pair}.
\begin{prop}
	Let $(L,\mu_{1},\mu_{2},\delta_{1},\delta_{2})$ be a \textbf{compatible AssDer pairs}. Let be $T :L \rightarrow L$ an endomorphism operator such that, for all $x,y \in L$ and $i\in \{1,2\}$ 
	\begin{eqnarray}
		T^2x&=&Tx, \label{endomorphism 1}  \\
		T\circ \delta_{\mathrm{i}}&=&\delta_{\mathrm{i}} \circ T. \label{endomorphism 2}
	\end{eqnarray}
	Then $(L,[\cdot,\cdot]_1,[\cdot,\cdot]_2,\delta_1,\delta_2)$ is a \textbf{compatible LieDer pairs} where,
	\begin{equation}
		[x,y]_{\mathrm{i}}=\mu_{\mathrm{i}}(Tx,y)-\mu_{\mathrm{i}}(Ty,x) \label{endomorphism 3}
	\end{equation}
\end{prop}
\begin{proof}
	For $i\in\{1,2\}$ and $x,y,z \in L$. \\
	We prove the previous results in two steps. Initially, we demonstrate that $(L,[\cdot,\cdot]_{\mathrm{i}},\delta_{\mathrm{i}})$ forms a LieDer pair.\\
	
	First, we verify that the skew symmetry holds by showing that $[x,y]_{\mathrm{i}} = -[y,x]_{\mathrm{i}}$.
	
	Next, through straightforward calculations, and leveraging the facts that $(L,\mu)$ is an associative algebra and that $T$ is an endomorphism operator, we establish that $(L,[\cdot,\cdot]_{\mathrm{i}})$ is a Lie algebra.
	
	Therefore, it remains to confirm that $\delta_{\mathrm{i}}$ indeed operates as a derivation on $(L,[\cdot,\cdot]_{\mathrm{i}})$.
	\begin{align*}
		\delta_{\mathrm{i}}([x,y]_{\mathrm{i}})&=\delta_{\mathrm{i}}(\mu_{\mathrm{i}}(Tx,y)-\mu_{\mathrm{i}}(Ty,x)) \\
		&=\delta_{\mathrm{i}} \circ \mu_{\mathrm{i}}(Tx,y)-\delta_{\mathrm{i}} \circ (Ty,x) \\
		&=\mu_{\mathrm{i}}(\delta_{\mathrm{i}} \circ T x,y)+\mu_{\mathrm{i}}(Tx,\delta_{\mathrm{i}}y)-\mu_{\mathrm{i}}(\delta_{\mathrm{i}} \circ Ty,x)-\mu_{\mathrm{i}}(Ty,\delta_{\mathrm{i}}x) \\
		&\overset{\eqref{endomorphism 2}}{=}\mu_{\mathrm{i}}(T \circ \delta_{\mathrm{i}} x,y)-\mu_{\mathrm{i}}(T \circ \delta_{\mathrm{i}} y,x)+\mu_{\mathrm{i}}(Tx,\delta_{\mathrm{i}}y)-\mu_{\mathrm{i}}(Ty,\delta_{\mathrm{i}}x) \\
		&=[\delta_{\mathrm{i}}x,y]_{\mathrm{i}}+[x,\delta_{\mathrm{i}}y]_{\mathrm{i}}
	\end{align*}
	This implies that $(L,[\cdot,\cdot]_{\mathrm{i}},\delta_{\mathrm{i}})$ constitutes a LieDer pair for all  $i\in \{1,2\}$.\\
	In the second step, we proceed to establish the compatibility of $(L,[\cdot,\cdot]_1,\delta_1)$ and $(L,[\cdot,\cdot]_2,\delta_2)$.
	\begin{align*}
		\circlearrowleft_{x,y,z}[x,[y,z]_1]_2&=[x,[y,z]_1]_2+[y,[z,x]_1]_2+[z,[x,y]_1]_2 \\
		&=[x,\mu_1(Ty,z)-\mu_1(Tz,y)]_2+[y,\mu_1(Tz,x)-\mu_1(Tx,z)]_2+[z,\mu_1(Tx,y)-\mu_1(Ty,x)]_2 \\
		&=\mu_2(Tx,\mu_1(Ty,z))-\mu_2(T \circ \mu_1(Ty,z),x)-\mu_2(Tx,\mu_1(Tz,y))+\mu_2(T\circ \mu_1(Tz,y),x) \\
		&\mu_2(Ty,\mu_1(Tz,x))-\mu_2(T\circ\mu_1(Tz,x),y)-\mu_2(Ty,\mu_1(Tx,z))+\mu_2(T\circ\mu_1(Tx,z),y) \\
		&+\mu_2(Tz,\mu_1(Tx,y))-\mu_2(T\circ\mu_1(Tx,y),z)-\mu_2(Tz,\mu_1(Ty,x))+\mu_2(T\circ\mu_1(Ty,x),z)
	\end{align*}
	Similarly we obtain 
	\begin{align*}
		\circlearrowleft_{x,y,z}[x,[y,z]_2]_1&=\mu_1(Tx,\mu_2(Ty,z))-\mu_1(T \circ \mu_2(Ty,z),x)-\mu_1(Tx,\mu_2(Tz,y))+\mu_1(T\circ \mu_2(Tz,y),x) \\
		&\mu_1(Ty,\mu_2(Tz,x))-\mu_1(T\circ\mu_2(Tz,x),y)-\mu_1(Ty,\mu_2(Tx,z))+\mu_1(T\circ\mu_2(Tx,z),y) \\
		&+\mu_1(Tz,\mu_2(Tx,y))-\mu_1(T\circ\mu_2(Tx,y),z)-\mu_1(Tz,\mu_2(Ty,x))+\mu_1(T\circ\mu_2(Ty,x),z)
	\end{align*}
	So 
	\small{
		\begin{align*}
			\circlearrowleft_{x,y,z}[x,[y,z]_1]_2+\circlearrowleft_{x,y,z}[x,[y,z]_2]_1&=\mu_2(Tx,\mu_1(Ty,z))-\mu_2(T \circ \mu_1(Ty,z),x)-\mu_2(Tx,\mu_1(Tz,y))+\mu_2(T\circ \mu_1(Tz,y),x) \\
			&\mu_2(Ty,\mu_1(Tz,x))-\mu_2(T\circ\mu_1(Tz,x),y)-\mu_2(Ty,\mu_1(Tx,z))+\mu_2(T\circ\mu_1(Tx,z),y) \\
			&+\mu_2(Tz,\mu_1(Tx,y))-\mu_2(T\circ\mu_1(Tx,y),z)-\mu_2(Tz,\mu_1(Ty,x))+\mu_2(T\circ\mu_1(Ty,x),z)\\
			&+\mu_1(Tx,\mu_2(Ty,z))-\mu_1(T \circ \mu_2(Ty,z),x)-\mu_1(Tx,\mu_2(Tz,y))+\mu_1(T\circ \mu_2(Tz,y),x) \\
			&\mu_1(Ty,\mu_2(Tz,x))-\mu_1(T\circ\mu_2(Tz,x),y)-\mu_1(Ty,\mu_2(Tx,z))+\mu_1(T\circ\mu_2(Tx,z),y) \\
			&+\mu_1(Tz,\mu_2(Tx,y))-\mu_1(T\circ\mu_2(Tx,y),z)-\mu_1(Tz,\mu_2(Ty,x))+\mu_1(T\circ\mu_2(Ty,x),z) \\
			&\overset{\eqref{endomorphism 1}}{=}\mu_2(\mu_1(Tx,Ty),z)+\mu_1(\mu_2(Tx,Ty),z)-\mu_2(\mu_1(Tx,Tz),y)-\mu_1(\mu_2(Tx,Tz),y) \\
			&+\mu_2(\mu_1(Ty,Tz),x)+\mu_1(\mu_2(Ty,Tz),x)-\mu_2(\mu_1(Ty,Tx),z)-\mu_1(\mu_2(Ty,Tx),z) \\
			&+\mu_2(\mu_1(Tz,Tx),y)+\mu_1(\mu_2(Tz,Tx),y)-\mu_2(\mu_1(Tz,Ty),x)-\mu_1(\mu_2(Tz,Ty),x) \\
			&-\mu_2(\mu_1(Ty,Tz),x)+\mu_2(\mu_1(Tz,Ty),x)-\mu_2(\mu_1(Tz,Tx),y)+\mu_2(\mu_1(Tx,Tz),y) \\
			&-\mu_2(\mu_1(Tx,Ty),z)+\mu_2(\mu_1(Ty,Tx),z)-\mu_1(\mu_2(Ty,Tz),x)+\mu_1(\mu_2(Tz,Ty),x)\\
			&-\mu_1(\mu_2(Tz,Tx),y)-\mu_1(\mu_2(Tx,Tz),y)-\mu_1(\mu_2(Tx,Ty),z)+\mu_1(\mu_2(Ty,Tx),z)\\
			&=0.
		\end{align*}
	}
	Finally, 
	\begin{align*}
		\delta_1([x,y]_2)+\delta_2([x,y]_1)&=\delta_1(\mu_2(Tx,y)-\mu_2(Ty,x))+\delta_2(\mu_1(Tx,y)-\mu_1(Ty,x)) \\
		&=(\delta_1\circ\mu_2(Tx,y)+\delta_2\circ\mu_1(Tx,y))-(\delta_1\circ\mu_2(Ty,x)+\delta_2\circ\mu_1(Ty,x)) \\
		&\text{by }\eqref{definition CADR assertion 2} \text{ and equation }\eqref{endomorphism 2} \\
		&=\mu_2(T\circ \delta_1x,y)+\mu_2(Tx,\delta_1y)+\mu_1(T\circ \delta_2x,y)+\mu_1(Tx,\delta_2y)\\
		&-\mu_2(T\circ\delta_1y,x)-\mu_2(Ty,\delta_1x)-\mu_1(T\circ \delta_2y,x)-\mu_1(Ty,\delta_2x) \\
		&=[\delta_1x,y]_2+[x,\delta_1y]_2+[\delta_2x,y]_1+[x,\delta_2y]_1
	\end{align*}
	This complete the proof.
\end{proof}
So we obtain another diagram
\begin{center}
	$\xymatrix{
		\textbf{\emph{compatible AssDer pairs }}&\underrightarrow{ \text{endomrophism operator T where }T^2=T}  &\textbf{\emph{compatible LieDer pairs}}
	}$
\end{center}
\begin{prop}
	Let be  $(L,\prec_{1},\succ_{1},\prec_{2},\succ_{2},\delta_{1},\delta_{2})$ a \textbf{compatible dendriDer pair}. Define the bilinear maps:
	\begin{eqnarray}
		\mu_{\mathrm{i}}&:&L\times L \longrightarrow L,\quad \text{for} \text{by}:\\
		&&\mu_{\mathrm{i}}(x,y)=x\prec_{\mathrm{i}}y+x\succ_{\mathrm{i}}y.
	\end{eqnarray}
	Then $(L,\mu_{1},\mu_{2},\delta_{1},\delta_{2})$ is a \textbf{compatible AssDer pair}.
\end{prop}
\begin{proof}
	Building upon Proposition \eqref{prop dendriDer to assDer}, we establish that both $(L,\mu_{1},\delta_{1})$ and $(L ,\mu_{2},\delta_{2})$ are indeed Ass Der pairs.
	
	Now, we must prove the two identities outlined in Definition \eqref{compa AssDer pairs}. It's worth noting that the first identity has already been demonstrated by Apurba Das, Mabrouk Sami, and Chtioui Taoufik in their work \cite{chtioui2021co}. Therefore, so we need just to prove the second identity, denoted as \eqref{definition CADR assertion 2}, by using \eqref{comp dendrDer 2} and \eqref{comp dendrDer 3}, as follows:
	\begin{eqnarray*}
		\delta_{1}\circ\mu_{2}(x,y)+\delta_{2}\circ\mu_{1}(x,y)&=& \delta_{1}(x\prec_{2}y+ x\succ_{2}y)+\delta_{2}(x\prec_{1}y+ x\succ_{1}y)\\
		&=& \delta_{1}(x\prec_{2}y)+\delta_{1}( x\succ_{2}y)+\delta_{2}(x\prec_{1}y)+\delta_{2}( x\succ_{1}y)\\
		&=& \delta_{1}x\prec_{2}y+x\prec_{2}\delta_{1}y+\delta_{1}x\succ_{2}y+x\succ_{2}\delta_{1}y+\delta_{2}x\prec_{1}y\\ 
		&+& x\prec_{1}\delta_{2}y	 +\delta_{2}x\succ_{1}y+x_{1}\succ_{1}\delta_{2}y\\
		&=& (\delta_{1}x\prec_{2}y+\delta_{1}x\succ_{2}y)+(x\prec_{2}\delta_{1}y+x\succ_{2}\delta_{1}y)+(\delta_{2}x\prec_{1}y+\delta_{2}x\succ_{1} y)\\
		&+& (x\prec_{1}\delta_{2}y+x\succ_{1}\delta_{2}y)\\
		&=& \mu_{2}(\delta_{1}x,y)+\mu_{2}(x,\delta_{1}y)+\mu_{1}(\delta_{2}x,y)+\mu_{1}(x,\delta_{2}y).
	\end{eqnarray*}
	This complete the proof.
\end{proof}
Now, we define the \textbf{compatible pre-LieDer pair}. First recall that a pre-LieDer pair is a pre-Lie algebra, see \cite{baxter1960analytic} for more details, equipped with a derivation.
\begin{defn} \label{def comp pre-LieDer pairs}
	A \textbf{compatible pre-LieDer pair} $(L,\circ_{1},\circ_{2},\delta_{1},\delta_{2})$ where  $L$ is a vector space, $\circ_{1}$ and $\circ_{2}$  are two pre-Lie algebra structures on $L$  and $\delta_{1}$ and $\delta_{2}$ are two derivations on $L$ such that, for all $x,y,z \in L$
	\begin{equation}
		x\circ_{1}(y\circ_{2} z)+x\circ_{2}(y\circ_{1} z)-(x\circ_{2}y)\circ_{1}z - (x\circ_{1} y )\circ_{2} z= y\circ_{1}(x\circ_{2} z)+y\circ_{2}(x\circ_{1} z)
		-(y\circ_{2}x)\circ_{1}z - (y\circ_{1}x )\circ_{2} z,\label{eq compa pre-LieDer 1}
	\end{equation} 
	\begin{equation}
		\delta_{1}(x\circ_{2} y)+\delta_{2}(x\circ_{1} y)~~=~~\delta_1x\circ_2y+x\circ_2\delta_1y+\delta_2x\circ_1y+x\circ_1\delta_2y.\label{eq compa pre-LieDer 2}
	\end{equation}
	
\end{defn}	
\begin{prop}
	A $5$-uple $(L,\circ_{1},\circ_{2},\delta_{1},\delta_{2})$ is a  \textbf{compatible pre-LieDer pair} if and only if $ \circ_{1}$ and $\circ_{2}$ are pre-Lie structures on $L$, such that for all $k_{1},k_{2}\in\mathbb{K}$ the following operations, for $x,y\in L$
	\begin{align*}
		x\circ y&=k_{1}x\circ_{1}y+k_{2}x\circ_{2}y, \\
		\delta &=p_{1}\delta_{1}+p_{2}\delta_{2} .
	\end{align*}
	
	Defines a pre-Lie Der pair structure on $L$.
\end{prop}

\begin{defn}
	Let $(L,\circ_{1},\circ_{2},\delta_{1},\delta_{2})$ and $(H,\circ^{'}_{1},\circ^{'}_{2},\delta^{'}_{1},\delta^{'}_{2})$ be two \textbf{compatible pre-LieDer pairs}. A homomorphism of \textbf{compatible pre-LieDer pairs}. $\varphi:L\longrightarrow H$ is both a pre-Lie Der pair homomorphism from $(L,\circ_{1},\delta_{1})$ to $(H,\circ^{'}_{1},\delta^{'}_{1})$  and  a pre-Lie Der pair homomorphism from $(L,\circ_{2},\delta_{2})$ to $(H,\circ^{'}_{2},\delta^{'}_{2})$.
\end{defn}
\begin{re}
	The previous definition means that, for $ i \in \{1,2\}$ and $x,y \in L$
	\begin{eqnarray}
		&&	\label{eq3.5}\varphi (x \circ_{\mathrm{i}} y)= \varphi(x)\circ'_{\mathrm{i}}\varphi(y), \\
		&& \label{eq3.6}\varphi\circ \delta_{\mathrm{i}}=\delta'_{\mathrm{i}}\circ\varphi.
	\end{eqnarray}
\end{re}

\begin{prop}
	Let be $(L,\prec_{1},\succ_{1},\prec_{2},\succ_{2},\delta_{1},\delta_{2})$ a compatible \textbf{dendriDer pair}. Define the bilinear maps:
	\begin{eqnarray*}
		\circ_{\mathrm{i}}&:& L\times L\longrightarrow L,\quad  \text{by}\\
		x\circ_{\mathrm{i}} y&=&x\succ_{\mathrm{i}} y -y\prec_{\mathrm{i}} x,\quad \text{for } i\in\{1,2\} \text{ and } x,y\in L.
	\end{eqnarray*}
	Then $(L,\circ_{1},\circ_{2},\delta_{1},\delta_{2})$ is a \textbf{compatible pre-LieDer pair}.
\end{prop}
\begin{proof}
	In the previous discussion, we established that both $(L,\circ_{1},\delta_{1})$ and $(L,\circ_{2},\delta_{2})$ are indeed pre-LieDer pairs. Our next objective is to verify the two identities laid out in Definition \eqref{def comp pre-LieDer pairs}.
	
	The first identity is straightforward to confirm. Therefore, let's concentrate on proving the second identity. Consider for all $x,y\in L$. By using \eqref{comp dendrDer 2} and \eqref{comp dendrDer 3}, we can demonstrate that:
	\begin{eqnarray*}
		\delta_{1}(x\circ_{2} y)+\delta_{2}(x\circ_{1} y)&=& \delta_{1}(x\succ_{2}-y\prec_{2}x)+ \delta_{2}(x\succ_{1} y -y\prec_{1}x)\\
		&=& \delta_{1}x\succ_{2}y+x\succ_{2}\delta_{1}y-\delta_{1}y\prec_{2}x-y\prec_{2}\delta_{1}x+\delta_{2}x\succ_{1}y\\
		&+&x\succ_{1}\delta_{2}y-\delta_{2}y\prec_{1}x-y\prec_{1}\delta_{2}x\\
		&=& \delta_{1}x\circ_{2}y+x\circ_{2}\delta_{1}y+\delta_{2}x\circ_{1}y+x\circ_{1}\delta_{2}y.
	\end{eqnarray*}
	This complete the proof.
\end{proof}


In the following proposition, we investigate the connection between compatible LieDer and pre-LieDer pairs through the commutator, defined as $[x,y]_c = x \circ y - y \circ x$. This exploration serves as a generalization of the following diagram:
\begin{center}
	$\xymatrix{
		\textbf{\emph{pre-LieDer pairs }}&\underrightarrow{ commutator} &\textbf{\emph{LieDer pairs}}
	}$
\end{center}	
\begin{prop}
	Let $(L,\circ_{1},\circ_{2},\delta_{1},\delta_{2})$ be a \textbf{compatible pre-LieDer pair}. Define for all $x,y \in L$
	\begin{align*}
		[x,y]_{c,1}&=x\circ_{1}y-y\circ_{1}x, \\
		[x,y]_{c,2}&=x\circ_{2}y-y\circ_{2}x.
	\end{align*}
	Then $(L,[\cdot,\cdot]_{c,1},[\cdot,\cdot]_{c,2},\delta_{1},\delta_{2})$ is a \textbf{compatible LieDer pair}.
\end{prop}
\begin{proof}
	Let be $x,y,z\in L,$
	\begin{align*}
		[x,y]_{c,1}&=x\circ_{1}y-y\circ_{1}x=-\left(y\circ_{1}x- x\circ_{1}y\right)=-[y,x]_{c,1},\\
		[x,y]_{c,2}&=x\circ_{2}y-y\circ_{2}x=-\left(y\circ_{2}x- x\circ_{2}y\right)=-[y,x]_{c,2}.
	\end{align*}
	
	which mean the skey symmetry of both $[\cdot,\cdot]_{c,1}$ and $[\cdot,\cdot]_{c,2}$ is satisfied\\
	It is easy to prove that, for all $i\in\{1,2\}$\\ $\circlearrowleft_{x,y,z}[x,[y,z]_{c,i}]_{c,i}=0$.\\
	Also
	\begin{align*}
		\delta_{i}\left( [x,y]_{c,i}\right)&= \delta_{i}\left( x\circ_{i}y-y\circ_{i}x\right)\\
		&= \delta_{i} x\circ_{i} y +x\circ_{i}\delta_{i} y -\delta_{i} y \circ_{i} x - y \circ_{i}\delta_{i} x,  ( \text{since } \delta_{i}  \text{ is a  derivation  on } (L,\circ_{i}))\\
		&= \left( \delta_{i}x\circ_{i}y-y\circ_{i}\delta_{i}x\right)+ \left( x\circ_{i}\delta_{i}y-\delta_{i}y\circ_{i}x\right)\\
		&= [\delta_{i}x,y]_{c,i}+[x,\delta_{i}]_{c,i}.
	\end{align*}
	which mean $(L,[\cdot,\cdot]_{c,i},\delta_{i})$ are two LieDer pairs, for $i \in \{1,2\}$. \\
	Let's now  focus on proving that
	\begin{center}
		$\circlearrowleft_{x,y,z}[x,[y,z]_{c,1}]_{c,2}+\circlearrowleft_{x,y,z}[x,[y,z]_{c,2}]_{c,1}=0$.
	\end{center}
	\begin{eqnarray*}	
		[x,[y,z]_{c,1}]_{c,2}&=& x\circ_{2}[y,z]_{c,1}-[y,z]_{c,1}\circ_{2}x\\
		&=& x\circ_{2}(y\circ_{1}z)-x\circ_{2}(z\circ_{1}y)-(y\circ_{1}z)\circ_{2}x+(z\circ_{1}y)\circ_{2}x.
	\end{eqnarray*} 
	Similarly 
	\begin{align*}
		[y,[z,x]_{c,1}]_{c,2}&=y\circ_{2}(z\circ_{1}x)-y\circ_{2}(x\circ_{1}z)-(z\circ_{1}x)\circ_{2}y+(x\circ_{1}z)\circ_{2}y,\\
		[z,[x,y]_{c,1}]_{c,2}&=z\circ_{2}(x\circ_{1}y)-z\circ_{2}(y\circ_{1}x)-(x\circ_{1}y)\circ_{2}z+(y\circ_{1}x)\circ_{2}z,\\
		[x,[y,z]_{c,2}]_{c,1}&=x\circ_{1}(y\circ_{2}z)-x\circ_{1}(z\circ_{2}y)-(y\circ_{2}z)\circ_{1}x+(z\circ_{2}y)\circ_{1}x, \\
		[y,[z,x]_{c,2}]_{c,1}&=y\circ_{1}(z\circ_{2}x)-y\circ_{1}(x\circ_{2}z)-(z\circ_{2}x)\circ_{1}y+(x\circ_{2}z)\circ_{1}y,\\
		[z,[x,y]_{c,2}]_{c,1}&=z\circ_{1}(x\circ_{2}y)-z\circ_{1}(y\circ_{2}x)-(x\circ_{2}y)\circ_{1}z+(y\circ_{2}x)\circ_{1}z.
	\end{align*}
	
	Now, by using the fact that $(L,\circ_{1},\circ_{2},\delta_{1},\delta_{2})$ is \textbf{compatible pre-LieDer pair} we obtain
	\begin{equation*}
		\circlearrowleft_{x,y,z}[x,[y,z]_{c,1}]_{c,2}+\circlearrowleft_{x,y,z}[x,[y,z]_{c,2}]_{c,1}=0.
	\end{equation*}
	on the other hand
	\begin{align*}
		\delta_1[x,y]_{c,2}+\delta_2[x,y]_{c,1}&=\delta_1(x\circ_2y-y\circ_2x)+\delta_2(x\circ_1y-y\circ_1x) \\
		&=(\delta_1(x\circ_2y)+\delta_2(x\circ_1y))-(\delta_1(y\circ_2x)+\delta_2(y\circ_1x)) \\
		&\overset{\eqref{eq compa pre-LieDer 2}}{=}\delta_1x\circ_2y+x\circ_2\delta_1y+\delta_2x\circ_1y+x\circ_1\delta_2y-\delta_1y\circ_2x-y\circ_2\delta_1x-\delta_2y\circ_1x-y\circ_1\delta_2x \\
		&=[\delta_1x,y]_{c,2}+[x,\delta_1y]_{c,2}+[\delta_2x,y]_{c,1}+[x,\delta_2y]_{c,1}
	\end{align*}
	This complete the proof.
\end{proof}
So we obtain a new diagram as follow
\begin{center}
	$\xymatrix{
		\textbf{\emph{compatible pre-lieDer pairs }}&\underrightarrow{ commutator} &\textbf{\emph{compatible LieDer pairs}}
	}$
\end{center}
In the next proposition, we examine the passage from \textbf{compatible LieDer pairs} to \textbf{compatible pre-LieDer pairs} by using the Rota-Baxter operator with weight $\lambda = 0$. This process extends the following diagram:
\begin{center}
	$\xymatrix{
		\textbf{\emph{LieDer pairs }}&\underrightarrow{ R.B.O} &\textbf{\emph{pre-LieDer pairs}}
	}$
\end{center}
\begin{prop}
	Let $(L,[\cdot,\cdot]_{1},[\cdot,\cdot]_{2},\delta_{1},\delta_{2})$ be a \textbf{compatible LieDer pairs} and  $R: L\rightarrow L $ be a Rota Baxter operator of weight $\lambda=0$, such that for $i \in \{1,2\} \text{ and } x,y \in L$
	\begin{eqnarray}
		R\circ\delta_{\mathrm{i}}&=&\delta_{\mathrm{i}}\circ R, \label{eq3.7} \\
		x\circ_{\mathrm{i}} y &=& [Rx,y]_{\mathrm{i}}. \label{eq3.8}
	\end{eqnarray}  
	Then $(L,\circ_{1},\circ_{2},\delta_{1},\delta_{2})$ is a \textbf{compatible pre-LieDer pairs}. 
\end{prop}
\begin{proof}
	First, let's prove that $(L,\circ_{1},\delta_{1})$ (repectively $(L,\circ_{2},\delta_{2})$) is a pre-Lie Der pair. It's evident that $(L,\circ_{1})$ (respectively $(L,\circ_{2})$) forms a pre-Lie algebra. Consequently, our objective is to confirm that $\delta_{\mathrm{i}}$ is a derivation on $(L,\circ_{\mathrm{i}})$, where $i\in\{1,2\}$.
	\begin{align*}
		\delta_{\mathrm{i}}(x\circ_{\mathrm{i}}y)&= \delta_{\mathrm{i}}[Rx,y]_{\mathrm{i}}\\
		&=  [\delta_{\mathrm{i}} \circ Rx,y]_{\mathrm{i}} + [Rx,\delta_{\mathrm{i}}y]_{\mathrm{i}}\\
		&\overset{\eqref{eq3.7}}{=}  [R\circ\delta_{\mathrm{i}} x,y]_{\mathrm{i}}+[R x, \delta_{\mathrm{i}}y]_{\mathrm{i}}\\
		&= \delta_{\mathrm{i}}x\circ_{\mathrm{i}}y+x\circ_{\mathrm{i}}\delta_{\mathrm{i}} y.
	\end{align*}
	Which mean $(L,\circ_{1},\delta_{1})$ is a pre-Lie Der pair (resp $(L,\circ_{2},\delta_{2})$ is a pre-Lie Der pair).\\
	Second step,
	\begin{align*}
		x\circ_{1}(y\circ_{2}z)&=[Rx,y\circ_{2}z]_{1}=[Rx,[Ry,z]_{2}]_{1}\\
		&\text{Similarly}\\
		x\circ_{2}(y\circ_{1}z)&=[Rx,[Ry,z]_{1}]_{2}\\
		(x\circ_{2}y)\circ_{1}z&=[R[Rx,y]_{2},z]_{1}\\
		(x\circ_{1}y)\circ_{2}z&=[R[Rx,y]_{1},z]_{2}\\
	\end{align*}
	
	So
	\begin{align*}
		&x\circ_{1}(y\circ_{2}z)+x\circ_{2}(y\circ_{1}z)-(x\circ_{2}y)\circ_{1}z-(x\circ_{1}y)\circ_{2}z\\ &=[Rx,[Ry,z]_{2}]_{1}+[Rx,[Ry,z]_{1}]_{2}-[R[Rx,y]_{2},z]_{1} - [R[Rx,y]_{1},z]_{2}\\
		&=[[Rx,Ry]_{2},z]_{1}+[Ry,[Rx,z]_{2}]_{1}+[[Rx,Ry]_{1},z]_{2}+[Ry,[Rx,z]_{1}]_{2}
		-[R[Rx,y]_{2},z]_{1}-[R[Rx,y]_{1},z]_{2}\\
		&=[[Rx,Ry]_{2}-R[Rx,y]_{2},z]_{1}
		+ [[Rx,Ry]_{1}-R[Rx,y]_{1},z]_{2}
		+[Ry,[Rx,z]_{2}]_{1}+ [Ry,[Rx,z]_{1}]_{2}\\
		&=-[R[Ry,x]_{2},z]_{1}-[R[Ry,x]_{1},z]_{2}
		+ [Ry,[Rx,z]_{2}]_{1}+[Ry,[Rx,z]_{1}]_{2}\\
		&= y\circ_{1}(x\circ_{2}z)+y\circ_{2}(x\circ_{1}z)-(y\circ_{2}x)\circ_{1}z
		-(y\circ_{1}x)\circ_{2}z.
	\end{align*}
	
	Third step, 
	\begin{align*}
		\delta_1(x\circ_2y)+\delta_2(x\circ_2y)&=\delta_1[Rx,y]_2+\delta_2[Rx,y]_1\\
		&\overset{\eqref{Def CLDP 1}}{=}[\delta_1\circ Rx,y]_2+[Rx,\delta_1y]_2+[\delta_2\circ Rx,y]_1+[Rx,\delta_2y]_1 \\
		&\overset{\eqref{eq3.7}}{=} [R\circ \delta_1x,y]_2+[Rx,\delta_1y]_2+[R\circ \delta_2x,y]_1+[Rx,\delta_2y]_1 \\
		&=\delta_1x\circ_2y+x\circ_2\delta_1y+\delta_2x\circ_1y+x\circ_1\delta_2y
	\end{align*}
	This complet the proof.
\end{proof}
\begin{prop} \label{comp zinDer to comp AssDer}
	Let be $(L,\ast_{1},\ast_{2},\delta_{1},\delta_{2})$ a \textbf{compatible zinbiel Der pair}. define two bilinear maps:
	\begin{eqnarray*}
		\mu_{\mathrm{i}}&:& L\times L \longrightarrow L \quad \text{ by} \\
		&&\mu_{\mathrm{i}}(x,y)=x\ast_{\mathrm{i}}y+y\ast_{\mathrm{i}}x,\quad  \forall i\in\{1,2\}\text{ and } \forall x,y\in L.
	\end{eqnarray*}
	Then $(L,\mu_{1},\mu_{2},\delta_{1},\delta_{2})$  is a \textbf{compatibles AssDer pair}.
\end{prop}

We need the folowing proposition in the proof of the proposition \eqref{comp zinDer to comp AssDer}.
\begin{prop}
	Let $(L,\ast_{1},\ast_{2},\delta_{1},\delta_{2})$ be a \textbf{compatible zinbiel Der pair}. Then the following equation holds
	\begin{equation} \label{eq*}
		x\ast_{1}(z \ast_{2} y)+x\ast_{2}(z \ast_{1}y)=z\ast_{1}(x \ast_{2} y)+z\ast_{2}(x \ast_{1}y), \ \ \ \forall x,y,z \in L
	\end{equation}
\end{prop}

\begin{proof}
	Let $x,y,z \in L$ 
	\begin{align*}
		x\ast_{1}(z \ast_{2} y)+x\ast_{2}(z \ast_{1}y)&=(x\ast_{1}z)\ast_{2}y+(x\ast_{2}z)\ast_{1}y+(z\ast_{1}x)\ast_{2}y+(z\ast_{2}x)\ast_{1}y \\
		&=(z\ast_{1}x)\ast_{2}y+(z\ast_{2}x)\ast_{1}y+(x\ast_{1}z)\ast_{2}y+(x\ast_{2}z)\ast_{1}y \\
		&=z\ast_{1}(x \ast_{2} y)+z\ast_{2}(x \ast_{1}y)
	\end{align*}
\end{proof}
\begin{re}
	The equation \eqref{eq*} still true in the case of compatible zinbiel algebras too $(L,\ast_{1},\ast_{2})$.
\end{re}
\begin{proof}
	By proposition \eqref{prop zinDer to AssDer pair} we proved that $(L,\mu_1,\delta_1)$ and $(L,\mu_2,\delta_2)$ are two AssDer pairs. So by definition \eqref{compa AssDer pairs} we need just to prove that $ \forall x,y,z\in L$
	\begin{description}
		\item[(i)] $\mu_{2}(x,\mu_{1}(y,z))+\mu_{1}(x,\mu_{2}(y,z))=\mu_{2}(\mu_{1}(x,y),z)+\mu_{1}(\mu_{2}(x,y),z).$
		\item[(ii)] $\delta_{i}(\mu_{j}(x,y))=\mu_{j}(\delta_{i}x,y)+\mu_{j}(x,\delta_{i}y)$
	\end{description}
	We begin by the first assertion,
	\begin{align*}
		\mu_2(x,\mu_1(y,z))+\mu_1(x,\mu_2(y,z))&=x \ast_{2}(y\ast_{1}z)+x\ast_{2}(z\ast_{1}y)+(y\ast_{1}z)\ast_{2}x+(z\ast_{1}y)\ast_{2}x \\
		&+x \ast_{1}(y\ast_{2}z)+x\ast_{1}(z\ast_{2}y)+(y\ast_{2}z)\ast_{1}x+(z\ast_{2}y)\ast_{1}x
	\end{align*}
	By definition of \textbf{compatible zinbiel Der pair} we have 
	\begin{align*}
		x\ast_2(y\ast_1z)+x\ast_1(y\ast_2z)&=(x\ast_1y)\ast_2z+(y\ast_1x)\ast_2z+(x\ast_2y)\ast_1z+(y\ast_2x)\ast_1z \\
		z\ast_2(y\ast_1x)+z\ast_1(y\ast_2x)&=(y\ast_1z)\ast_2x+(z\ast_1y)\ast_2x+(y\ast_2z)\ast_1x+(z\ast_2y)\ast_1x
	\end{align*}
	and using the  equation \eqref{eq*} we obtain 
	\begin{align*}
		\mu_2(x,\mu_1(y,z))+\mu_1(x,\mu_2(y,z))&=(x\ast_1y)\ast_2z+(y\ast_1x)\ast_2z+(x\ast_2y)\ast_1z+(y\ast_2x)\ast_1z \\
		&+z\ast_2(y\ast_1x)+z\ast_1(y\ast_2x)+z\ast_{1}(x \ast_{2} y)+z\ast_{2}(x \ast_{1}y) \\
		&=\mu_2(\mu_1(x,y),z)+\mu_1(\mu_2(x,y),z)
	\end{align*}
	For the second assertion 
	\begin{align*}
		\delta_1\circ \mu_2(x,y)+\delta_2\circ \mu_1(x,y)&=\delta_1(x\ast_2y+y\ast_2x)+\delta_2(x\ast_1y+y\ast_1x) \\
		&=\delta_1(x\ast_2y)+\delta(y\ast_2x)+\delta_2(x\ast_1y)+\delta_2(y\ast_1x) \\
		&\text{ by definition of compatible zinbiel Der pair } \\
		&=\delta_1x\ast_2y+x\ast_2\delta_1y+\delta_2x\ast_1y+x\ast_1\delta_2y \\
		&+\delta_1y\ast_2x+y\ast_2\delta_1x+\delta_2y\ast_1x+y\ast_1\delta_2x \\
		&=(\delta_1x\ast_2y+y\ast_2\delta_1x)+(x\ast_2\delta_1y+\delta_1y\ast_2x) \\
		&+(\delta_2x\ast_1y+y\ast_1\delta_2x)+(x\ast_1\delta_2y+\delta_2y\ast_1x)
	\end{align*}
	This complete the proof.
\end{proof}
\section{Cohomology of compatible AssDer pairs}	\label{section 5}
In this section we introduce the cohomology theory of \textbf{compatible AssDer pairs}.\\
We recall first the cohomology compatible associative algebras.

\paragraph*{Cohomology of compatible associative algebras: }
Let $\mathcal{A}=(A,\mu_1,\mu_2)$ be a compatible associative algebra, using the  compatible pair of Maurer-Cartan elements and the Gerstenhaber bracket, authors in \cite{chtioui2021co} defined the cohomology of compatible associative algebra as follow. Let, for  $\varphi \in C^{\mathrm{n}}(A,A)$
\begin{equation*}
	d^n_1(\varphi)=(-1)^{n-1}[\mu_1,\varphi]_G \text{ and } d^n_2(\varphi)=(-1)^{n-1}[\mu_2,\varphi]_G
\end{equation*}
be the coboundary maps respectively of $(A,\mu_1)$ and $(A,\mu_2)$. \\
Then $d_{\mathrm{c}}^n:\mathfrak{C}_{\mathrm{c}}^n(A,A)\rightarrow \mathfrak{C}_{\mathrm{c}}^{n+1}(A,A)$ is the coboundary map for the cohomology of the compatible associative algebra $\mathcal{A}$ with self coefficients which is given by, $(f_1,\ldots,f_n) \in \mathfrak{C}_{\mathrm{c}}^n(A,A), \quad \text{for } n \geq 1$,
\begin{equation*}
	d_c^n(f_1,\ldots,f_n)=(-1)^{n-1}\big([u_1,f_1]_G,\ldots,\underbrace{ [\mu_1,f_{i-1}]_G+[\mu_2,f_i]_G}_{\text{i-th place}},\ldots,\ldots,[\mu_2,f_n] \big),
\end{equation*}
where 
\begin{equation*}
	\mathfrak{C}_{\mathrm{c}}^n(A,A)=\underbrace{C^n(A,A)\oplus \ldots \oplus C^n(A,A)}_{\text{ n times }},\text{ and } \mathfrak{C}_{\mathrm{c}}^0(A,A)=\{y\in A;\quad \mu_1(x,y)-\mu_1(y,x)=\mu_2(x,y)-\mu_2(y,x),\quad \forall x\in A\}.
\end{equation*}


\subsection{Maurer-Cartan characterization of compatible AssDer pairs}
In \cite{das2020extensions}, authors showed that if $(\mu,\delta)\in C^2_{\mathrm{AssDer}}(A,A)$, then $(A,\mu,\delta)$ is an AssDer pair if and only if $(\mu,\delta)$ is a Maurer-Cartan element in the graded Lie algebra $(C^{\ast}_{\mathrm{AssDer}}(A,A)=\oplus_{n\geq 0}C^{n}_{\mathrm{AssDer}}(A,A),[\![\cdot,\cdot]\!])$ (means that, $[\![(\mu,\delta),(\mu,\delta)]\!]=0$).\\
Where $[\![\cdot,\cdot]\!]$ is defined as follow 
\begin{eqnarray*}
	[\![\cdot,\cdot]\!]&:&C^m_{\mathrm{AssDer}}(A,A) \times C^n_{\mathrm{AssDer}}(A,A) \rightarrow C^{m+n-1}_{\mathrm{AssDer}}(A,A),\\	&&[\![(f_m,f_{m-1}),(g_n,g_{n-1})]\!]:=\Big([f_m,g_n]_{\mathrm{G}},(-1)^{(m+1)}[f_m,g_{n-1}]_{\mathrm{G}}+[f_{m-1},g_n]_{\mathrm{G}} \big).
\end{eqnarray*}
With the diferential of $(A,\mu,\delta)$ with coefficient in itself is given by
\begin{equation*}
	d^n_{\mathrm{AssDer}}(f_n,g_{n-1})=(-1)^{n-1}[\![(\mu,\delta),(f_n,g_{n-1})]\!],\quad \text{for } (f_n,g_{n-1})\in C^n_{\mathrm{AssDer}}(A,A).
\end{equation*}
Recall that, from \cite{das2020extensions}, if $(A,w,\delta)$ is an AssDer pair, then $(C^{\ast}_{\mathrm{AssDer}}(A,A),[\![\cdot,\cdot]\!],d_{(w,\delta)}^{\ast})$ is a differnetial graded Lie algebra.
\begin{re}
	\begin{enumerate}
		\item [1)] If $(C^{\ast}_{\mathrm{AssDer}}(A,A),[\![\cdot,\cdot]\!])$ is a graded Lie algebra, then $(w,\delta)$ is the Maurer-Cartan element of it.
		\item [2)] Also we can see here that the differential $d^{\ast}_{\mathrm{AssDer}}$ coincids with $(d_{(w,\delta)}^{\ast})$.
	\end{enumerate}
\end{re}
Next, based on \cite{chtioui2021co,das2020extensions} we construct a new graded Lie algebra $(C^{\ast}_{\mathrm{AssDer}},[\![\cdot,\cdot]\!],d_{(w_1,\delta_1)}^{\ast},d_{(w_2,\delta_2)}^{\ast})$ and we introduce its Maure-Cartan element.
\begin{defn}
	Let $(C^{\ast}_{\mathrm{AssDer}}(A,A),[\![\cdot,\cdot]\!],d_{(w_1,\delta_1)}^{\ast})$ and $(C^{\ast}_{\mathrm{AssDer}}(A,A),[\![\cdot,\cdot]\!],d_{(w_2,\delta_2)}^{\ast})$ be two differential graded Lie algebras. We call the quadruple $(C^{\ast+1}_{\mathrm{AssDer}}(A,A),[\![\cdot,\cdot]\!],d_{(w_1,\delta_1)}^{\ast},d_{(w_2,\delta_2)}^{\ast})$ a bidifferntial graded Lie algebra if $d_{(w_1,\delta_1)}^{\ast}$ and $d_{(w_2,\delta_2)}^{\ast}$ satisfy 
	\begin{equation*}
		d_{(w_1,\delta_1)}^{\ast} \circ d_{(w_2,\delta_2)}^{\ast}+d_{(w_2,\delta_2)}^{\ast} \circ d_{(w_1,\delta_1)}^{\ast}=0.
	\end{equation*}
\end{defn}
Next we introduce the compatible pair Maurer-Cartan element.
\begin{defn}
	The pair $\big((w_1',\delta_1'),(w_2',\delta_2') \big)$ is called \textbf{ compatible pair Maurer-Cartan element} of the bidifferential graded Lie algebra $(C^{\ast}_{\mathrm{AssDer}}(A,A),[\![\cdot,\cdot]\!],d_{(w_1,\delta_1)}^{\ast},d_{(w_2,\delta_2)}^{\ast})$ if $(w_1',\delta_1') \text{ and } (w_2',\delta_2')$ are \textbf{ Maurer-Cartan elements} of the differential graded Lie algebras $(C^{\ast+1}_{\mathrm{AssDer}}(A,A),[\![\cdot,\cdot]\!],d_{(w_1,\delta_1)}^{\ast})$ and $(C^{\ast}_{\mathrm{AssDer}}(A,A),[\![\cdot,\cdot]\!],d_{(w_2,\delta_2)}^{\ast})$ respectively and the following equation holds 
	\begin{equation} \label{MC comp AssDer}
		d_{(w_1,\delta_1)}^{\ast}(w_1',\delta_1')+d_{(w_2,\delta_2)}^{\ast}(w_2',\delta_2')+[\![(w_1',\delta_1'),(w_2',\delta_2')]\!]=0.
	\end{equation}
\end{defn}
\begin{re}
	Since $(w,\delta)\in C^2_{\mathrm{AssDer}}(A,A)$ we have 
	\begin{eqnarray}
		[\![(w,\delta),(w,\delta)]\!]&=&([w,w]_G,-2[w,\delta]_G), \label{MC comp Ass1}
	\end{eqnarray}
	\begin{equation}
		[\![(w_1,\delta_1),(w_2,\delta_2)]\!]=([w_1,w_2]_G,-[w_1,\delta_2]_G+[\delta_1,w_2]_G). \label{MC comp Ass2}
	\end{equation}
\end{re}

\begin{thm} \label{Theorem CMC}
	Let $A$ be a vector space, $w_1,w_2 \in Hom(A\otimes A,A)$ and $\delta_1,\delta_{2} \in Hom(A,A)$. Then $(A,w_1,w_2,\delta_1,\delta_2)$ is a compatible AssDer pair if and only if $\big((w_1,\delta_1),(w_2,\delta_2) \big)$ is a \textbf{ compatible pair of Maurer-Cartan element} of the bidifferential graded Lie algebra $(C^{\ast}_{\mathrm{AssDer}},[\![\cdot,\cdot]\!],d_{(w_1,\delta_1)}^{\ast}=0,d_{(w_2,\delta_2)}^{\ast}=0)$.
\end{thm}
\begin{proof}
	The proof is easy to check, we use only from \cite{chtioui2021co} that $(A,w_1,w_2,\delta_1,\delta_2)$ is a compatible AssDer pair means that 
	\begin{equation*}
		[w_1,w_1]_{\mathrm{G}}=0,\quad [w_2,w_2]_{\mathrm{G}}=0,\quad \text{and } [w_1,w_2]_{\mathrm{G}}=0
	\end{equation*}
	And from equation \eqref{MC comp AssDer} and \eqref{MC comp Ass2} and the fact that $(A,w_1,\delta_1)$, $(A,w_2,\delta_2)$ are both an AssDer pairs we obtain 
	\begin{equation*}
		[w_1,\delta_2]_{\mathrm{G}}=[w_2,\delta_2]_{\mathrm{G}}=0
	\end{equation*}
	This complete the proof.
\end{proof}
\subsection{Cohomology of compatible AssDer pairs}
So let $(L,\mu_1,\mu_2,\delta_{1},\delta_{2})$ be a compatible AssDer pair, with $w_{1}(x,y)=\mu_1(x,y)$ and  $w_{2}(x,y)=\mu_2(x,y)$.\\
By theorem \eqref{Theorem CMC}, $((w_{1},\delta_{1}),(w_{2},\delta_{2}))$   is a   Maurer-Cartan  element of the bidifferential graded  Lie algebra $(\mathcal{C}_{AssDer}^{\ast}(A,A),[\![\cdot,\cdot]\!],d_{(w_1,\delta_1)}^{\ast}=0,d_{(w_2,\delta_2)}^{\ast}=0).$\\
Define the set of compatible LieDer pairs $0$-cochains to be $0$.\\
For $n\geq 1$, define the space of $n$-cochains $\mathcal{C}_{\mathrm{C.A.D}}^{n}(A,A)$ by 
\begin{equation*}
	\mathcal{C}_{C.A.D}^{n}(A,A)=\underbrace{C_{AssDer}^n(A,A)\oplus\cdots\oplus C_{AssDer}^n(A,A)}_{n-times}
\end{equation*}
Define
\begin{equation*}
	d_{\mathrm{C.A.D}}^{1}:\mathcal{C}_{\mathrm{C.A.D}}^{1}(A,A)\longrightarrow \mathcal{C}_{\mathrm{C.A.D}}^{2}(A,A)
\end{equation*} 
By
\begin{equation}\label{partial 1 CAD}
	d_{\mathrm{C.A.D}}^{1} (f)= \left(([w_{1},f]_{G},-[f,\delta_1]_{G}),([w_{2},f]_{G},-[f,\delta_2]_{G})\right) \ \ \  \forall f\in\mathcal{C}_{C.L.D}^{1}(A,A).
\end{equation}
For $n\geq 2$ and $2 \leq i \leq n$ define 
\begin{equation*}
	d_{\mathrm{C.A.D}}^{n}:\mathcal{C}_{\mathrm{C.A.D}}^{n}(A,A)\longrightarrow \mathcal{C}_{\mathrm{C.A.D}}^{n+1}(A,A).
\end{equation*} 
By
\small{
	\begin{equation}\label{partial n CAD}
		\begin{split}	
			d_{\mathrm{C.A.D}}^{n}\Big((f^1,g^1),(f^2,g^2),\cdots,(f^n,g^n) \Big)=(-1)^{n-1}\Big(([w_1,f^1]_{\mathrm{G}},-[w_1,g^1]_{\mathrm{G}}-[f^1,\delta^1]_{\mathrm{G}}),\cdots,\\
			\underbrace{([w_2,f^{i-1}]_{\mathrm{G}}+[w_1,f^i]_{\mathrm{G}},-[w_2,g^{i-1}]_{\mathrm{G}}-[w_1,g^i]_{\mathrm{G}}-[f^{i-1},\delta_2]_{\mathrm{G}}-[f^i,\delta_1]_{\mathrm{G}})}_{\text{i-th place}},\\
			\ldots,([w_2,f^n]_{\mathrm{G}},[w_2,g^n]_{\mathrm{G}}-[f^n,\delta_2]_{\mathrm{G}}) \Big).
		\end{split}
	\end{equation}
}
\begin{thm} \label{coboundary CAD}
	The map $\partial$ is a coboundary operator 
	\begin{equation*}
		d_{\mathrm{C.A.D}}^{n+1} \circ d_{\mathrm{C.A.D}}^{n}=0
	\end{equation*}
\end{thm}
\begin{proof}
	The proof is similar to the proof of theorem \eqref{thm5.6} where $[\cdot,\cdot]_{\mathrm{NR}}$ is replaced by $[\cdot,\cdot]_{\mathrm{G}}$ with subject to the changes in the calccul.
\end{proof}
\begin{defn}
	Let $(A,\mu_1,\mu_2,\delta_1,\delta_2)$ be a \textbf{compatible AssDer pair}. The cohomology of the cochains complex $(\mathcal{C}_{\mathrm{C.A.D}}^{\ast}(A,A),d_{\mathrm{C.A.D}}^{\ast})$ is called the cohomology of $(A,\mu_1,\mu_2,\delta_1,\delta_2)$, where the $n$-th cohomology group is denoted by $\mathcal{H}_{\mathrm{C.A.D}}^{n}(A,A)$.
\end{defn}
\section{Cohomology of compatible LieDer pairs} \label{section 6}
In this section, we construct the bidifferential graded Lie algebra whose \textbf{Maurer-Cartan elements} (where the differentials depends on the structure of LieDer pair) are \textbf{compatible LieDer pair}. Also we give a cohomology of it. Consider the degree of elements in $C^n(L,L)$ are defined to be $(n-1)$. 
\subsection{Maurer-Cartan characterization of compatible LieDer pairs}
Let's first recall some basis results that we need next. 
\paragraph*{Maurer-Cartan element of compatible Lie algebras :}

A \textbf{Maurer-Cartan element} of a differential graded Lie algebra $(L=\oplus_{i \in \mathbb{Z}}L_i,[\cdot,\cdot],d)$ is an element $P \in L_1$ such that 
\begin{equation*} 
	dP+\frac{1}{2}[P,P]=0
\end{equation*}
we recall too, that $(C^{\star}(L,L),[\cdot,\cdot]_{NR})$ is a graded Lie algebra see 
Remark that any differential Lie algebra with a zero differential is simply a graded Lie algebra.
\begin{defn}{\cite{liu2023maurer}}\\
	Let $(L,[\cdot,\cdot],\partial_1)$ and $(L,[\cdot,\cdot],\partial_2)$ be two differential graded Lie algebras. The quadruple $(L,[\cdot,\cdot],\partial_1,\partial_2)$ is called a bidifferential graded Lie algebra if $\partial_1$ and $\partial_2$ satisfy 
	\begin{equation*}
		\partial_1 \circ \partial_2 + \partial_2 \circ \partial_1=0
	\end{equation*}
\end{defn}
\begin{defn}{\cite{liu2023maurer}} \\
	A pair $(P_1,P_2) \in L_1\oplus L_1$ is called a \textbf{Maurer-Cartan element} of the bidifferential graded Lie algebra $(L,[\cdot,\cdot],\partial_1,\partial_2)$ if and only if $P_1$ and $P_2$ are \textbf{Maurer-Cartan elements} of the differential graded Lie algebras $(L,[\cdot,\cdot],\partial_1)$ and $(L,[\cdot,\cdot],\partial_2)$ respectively and the following equation holds 
	\begin{equation*}
		\partial_1P_2+\partial_2P_1+[P_1,P_2]=0
	\end{equation*} 
\end{defn}

\paragraph*{Maurer-Cartan element of compatible LieDer pairs : }
So, Now we get under way the main of the this subsection.
\begin{defn}
	Let $(DC^{\star}(L,L),\{\cdot,\cdot\},d_{(w_1,\delta_1)})$ and $(DC^{\star}(L,L),\{\cdot,\cdot\},d_{(w_2,\delta_2)})$ be two differential graded Lie algebras. we call the quadruple $(DC^{\star}(L,L),\{\cdot,\cdot\},d_{(w_1,\delta_1)},d_{(w_2,\delta_2)})$ a bidifferential graded Lie algebra if $d_{(w_1,\delta_1)}\text{ and }d_{(w_2,\delta_2)}$ satisfy 
	\begin{equation} 
		d_{(w_1,\delta_1)} \circ d_{(w_2,\delta_2)}+d_{(w_2,\delta_2)} \circ 	d_{(w_1,\delta_1)} =0
	\end{equation}
\end{defn}
\begin{defn}
	The couple $\big((w_1',\delta_1'),(w_2',\delta_2') \big)$ is called \textbf{ Maurer-Cartan element} of the bidifferential graded Lie algebra $(DC^{\star}(L,L),\{\cdot,\cdot\},d_{(w_1,\delta_1)},d_{(w_2,\delta_2)})$ if $(w_1',\delta_1') \text{ and } (w_2',\delta_2')$ are \textbf{ Maurer-Cartan elements} of the differential graded Lie algebras $(DC^{\star}(L,L),\{\cdot,\cdot\},d_{(w_1,\delta_1)})$ and $(DC^{\star}(L,L),\{\cdot,\cdot\},d_{(w_2,\delta_2)})$ respectively and the following equation holds 
	\begin{equation} 
		d_{(w_1,\delta_1)}(w_1',\delta_1')+d_{(w_2,\delta_2)}(w_2',\delta_2')+\{(w_1',\delta_1'),(w_2',\delta_2')\}=0.
	\end{equation}
\end{defn}
\begin{re}
	Since $w\in \mathrm{Hom}(\wedge^2L,L)$ and $\delta \in \mathrm{Hom}(L,L)$ then $(w,\delta) \in DC^1(L;L)$ and we have 
	\begin{equation}
		\{(w,\delta),(w,\delta)\}=\big([w,w]_{\mathrm{NR}},-2[w,\delta]_{\mathrm{NR}} \big).
	\end{equation}
	Similarly we have 
	\begin{equation}
		\{(w_1,\delta_1),(w_2,\delta_2)\}=\Big([w_1,w_2]_{\mathrm{NR}},-[w_1,\delta_2]_{\mathrm{NR}}-[w_2,\delta_1]_{\mathrm{NR}} \Big)
	\end{equation}
\end{re}
We know that if $(L,[\cdot,\cdot])$ is a graded Lie algebra, then it is obvious that $(L,[\cdot,\cdot],\partial_1=0,\partial_2=0)$ is a bidifferential graded Lie algebra. Which implies the next theorem 
\begin{thm}\label{thm5.4}
	Let $L$ be a vector space, $w_1,w_2 \in \mathrm{Hom}(\wedge^2L,L)$ and $\delta_1,\delta_{2} \in \mathrm{Hom}(L,l)$. Then $(L,w_1,w_2,\delta_1,\delta_2)$ is a compatible LieDer pair if and only if $\big((w_1,\delta_1),(w_2,\delta_2) \big)$ is a \textbf{ Maurer-Cartan element} of the bidifferential graded Lie algebra $(DC^{\star}(L,L),\{\cdot,\cdot\},d_{(w_1,\delta_1)}=0,d_{(w_2,\delta_2)}=0)$
\end{thm}
\begin{proof}
	Suppose that $(L,w_1,w_2,\delta_1,\delta_2)$ is a \textbf{compatible LieDer pair}. \\
	By the the first assertion $i)$ in definition \eqref{def2.1}
	we obtain $[w_1,w_2]_{\mathrm{NR}}=0$. \\
	On the other side we have 
	\begin{align*}
		[w_1,\delta_2]_{\mathrm{NR}}&=w_1\circ \delta_2(x,y)-\delta_2 \circ w_1(c,y) \\
		&=w_1(\delta_2x,y)-w_1(\delta_2y,x)-\delta_2\circ w_1(x,y) \\
		&=0
	\end{align*}
	Similarly we obtain
	\begin{align*}
		[w_2,\delta_1]_{\mathrm{NR}}=0
	\end{align*}
	This implies that $\big((w_1,\delta_1),(w_2,\delta_2)\big)$ is a \textbf{Maurer-Cartan element}. \\
	For the second sense. \\
	Suppose that $\big((w_1,\delta_1),(w_2,\delta_2)\big)$ is a \textbf{Maurer-Cartan element} which means 
	$[w_1,w_2]=0,[w_1,\delta_2]+[w_2,\delta_1]=0$. \\
	With simple calcul we obtain that $(L,w_1,w_2,\delta_1,\delta_2)$ is a \textbf{compatible LieDer pair}.
\end{proof}
\subsection{Cohomology of compatible LieDer pair}
In this subsection, we indroduce a cohomology theory of \textbf{compatible LieDer pair}.

So let $(L,[\cdot,\cdot]_{1},[\cdot,\cdot]_{2},\delta_{1},\delta_{2})$ be a \textbf{compatible LieDer pair}, with $w_{1}(x,y)=[x,y]_{1}$ and  $w_{2}(x,y)=[x,y]_{2}$.\\
By theorem \eqref{thm5.4}, $((w_{1},\delta_{1}),(w_{2},\delta_{2}))$   is a   Maurer-Cartan  element of the bidifferential graded  Lie algebra
\begin{equation*}
	(\mathcal{DC}^{\ast}(L,L),\{\cdot,\cdot \},d_{(w_1,\delta_1)}=0,d_{(w_2,\delta_2)}=0).
\end{equation*} 
Define the set of compatible LieDer pairs $0$-cochains to be $0$.\\
For $n\geq 1$, define the space of $n$-cochains $\mathcal{C}_{C.L.D}^{n}(L,L)$ by 
\begin{equation*}
	\mathcal{C}_{C.L.D}^{n}(L,L)=\underbrace{\mathcal{DC}^{n}(L,L)\oplus\cdots\oplus\mathcal{DC}^{n}(L,L)}_{n-times}
\end{equation*}
Define
\begin{equation*}
	\partial_{\mathrm{C.L.D.P}}^{1}:\mathcal{C}_{C.L.D}^{1}(L,L)\longrightarrow \mathcal{C}_{C.L.D}^{2}(L,L)
\end{equation*} 
By
\begin{equation}\label{eq5.4}
	\partial_{\mathrm{C.L.D.P}}^{1} (f)= \left(([w_{1},f]_{\mathrm{NR}},-[f,\delta_1]_{\mathrm{NR}}),([w_{2},f]_{\mathrm{NR}},-[f,\delta_2]_{\mathrm{NR}})\right),\quad \forall f\in\mathcal{C}_{\mathrm{C.L.D}}^{1}(L,L)
\end{equation}
and for $n\geq 2$ and for $2\leq i \leq n$ define 
\begin{equation*}
	\partial_{\mathrm{C.L.D.P}}^{n}:\mathcal{C}_{\mathrm{C.L.D}}^{n}(L,L)\longrightarrow \mathcal{C}_{\mathrm{C.L.D}}^{n+1}(L,L)
\end{equation*} 
By
\small{
	\begin{equation}
		\begin{split}
			\partial_{\mathrm{C.L.D.P}}^{n}\Big((f^1,g^1),(f^2,g^2),\cdots,(f^n,g^n) \Big)=(-1)^{n-1}\Big(([w_1,f^1]_{\mathrm{NR}},-[w_1,g^1]_{\mathrm{NR}}-[f^1,\delta^1]_{\mathrm{NR}}),\cdots,\\
			\underbrace{([w_2,f^{i-1}]_{\mathrm{NR}}+[w_1,f^i]_{\mathrm{NR}},-[w_2,g^{i-1}]_{\mathrm{NR}}-[w_1,g^i]_{\mathrm{NR}}-[f^{i-1},\delta_2]_{\mathrm{NR}}-[f^i,\delta_1]_{\mathrm{NR}})}_{\text{i-th place}},
			\cdots,([w_2,f^n]_{\mathrm{NR}},[w_2,g^n]_{\mathrm{NR}}-[f^n,\delta_2]_{\mathrm{NR}}) \Big)
		\end{split}
	\end{equation}
}
\begin{thm} \label{thm5.6}
	The map $\partial_{\mathrm{C.L.D.P}}^{\ast}$ is a coboundary operator, means that 
	\begin{equation*}
		\partial_{\mathrm{C.L.D.P}}^{n+1} \circ \partial_{\mathrm{C.L.D.P}}^{n}=0
	\end{equation*}
\end{thm}
We will, first, introduce the following proposition that we need later in the proof of the theorem.
\begin{prop}\label{prop5.7}
	Let $j\in \{1,2,...,n\}$, let $f^j \in \mathcal{C}^n(L,L)$, $w_1,w_2 \in \mathcal{C}^2(L,L)$ and $\delta_1,\delta_2 \in \mathcal{C}^n(L,L)$. \\
	By theorem \eqref{thm5.4}, if $((w_1,\delta_1),(w_2,\delta_2))$ is a Maurer-Cartan element of the graded Lie algebra then we have
	\begin{align*}
		[w_1,w_2]_{\mathrm{NR}}=[w_1,w_1]_{\mathrm{NR}}=[w_2,w_2]_{\mathrm{NR}}&=0 \\
		[w_1,\delta_1]_{\mathrm{NR}}=[w_2,\delta_2]_{\mathrm{NR}}=[w_1,\delta_2]_{\mathrm{NR}}=[w_2,\delta_1]_{\mathrm{NR}}&=0
	\end{align*}
	Then the following identities holds 
	\begin{enumerate}
		\item [1)] $[w_1,[w_1,f^j]_{\mathrm{NR}}]_{\mathrm{NR}}=0$ ,
		\item [2)] $[w_1,[f^j,\delta_1]_{\mathrm{NR}}]_{\mathrm{NR}}=[[w_1,f^j]_{\mathrm{NR}},\delta_1]_{\mathrm{NR}}$ ,
		\item [3)] $[w_2,[w_1,f^j]_{\mathrm{NR}}]_{\mathrm{NR}}+[w_1,[w_2,f^j]_{\mathrm{NR}}]_{\mathrm{NR}}+[w_1,[w_1,f^{j+1}]_{\mathrm{NR}}]_{\mathrm{NR}}=0$ ,
		\item [4)] $[w_2,[f^j,\delta_1]_{\mathrm{NR}}]_{\mathrm{NR}}+[w_1,[f^j,\delta_2]_{\mathrm{NR}}]_{\mathrm{NR}}+[w_1,[f^{j+1},\delta_1]_{\mathrm{NR}}]_{\mathrm{NR}}-[[w_1,f^j]_{\mathrm{NR}},\delta_2]_{\mathrm{NR}}-[[w_2,f^j]_{\mathrm{NR}},\delta_1]_{\mathrm{NR}}-[[w_1,f^{j+1}]_{\mathrm{NR}},\delta_1]_{\mathrm{NR}}=0$ .
	\end{enumerate}
\end{prop}
\begin{proof}
	By using the fact that $((w_1,\delta_1),(w_2,\delta_2))$ is a Maurer-Cartan element of the graded Lie algebra
	and the graded Jacobi identity we have 
	\begin{align*}
		[w_1,[w_1,f^j]_{\mathrm{NR}}]_{\mathrm{NR}}&=[[w_1,w_1]_{\mathrm{NR}},f^j]-[w_1,[w_1,f^j]_{\mathrm{NR}}]_{\mathrm{NR}} \\
		&=2[[w_1,w_1]_{\mathrm{NR}},f^j]_{\mathrm{NR}}=0
	\end{align*}
	For second identity 
	\begin{align*}
		[w_1,[f^j,\delta_1]_{\mathrm{NR}}]_{\mathrm{NR}}&=[[w_1,f^j]_{\mathrm{NR}},\delta_1]_{\mathrm{NR}}-(-1)^{n-1}[f^j,[w_1,\delta_1]_{\mathrm{NR}}]_{\mathrm{NR}} \\
		&=[[w_1,f^j]_{\mathrm{NR}},\delta_1]_{\mathrm{NR}}.
	\end{align*}
	For the third one and by using the first identiy we have 
	\begin{align*}
		[w_2,[w_1,f^j]_{\mathrm{NR}}]_{\mathrm{NR}}+[w_1,[w_2,f^j]_{\mathrm{NR}}]_{\mathrm{NR}}=[[w_2,w_1]_{\mathrm{NR}},f^j]_{\mathrm{NR}}=0
	\end{align*}
	For the last one we just use the second identity. This complete the proof.
\end{proof}
Let's now prove the theorem \eqref{thm5.6}.
\begin{proof}
	For $\big((f^1,g^1),(f^2,g^2),...,(f^n,g^n) \big) \in \mathcal{C}_{C.L.D}^{n}(L,L)$ and for $2 \leq i \leq n$
	\small{ 
		\begin{align*}
			&\partial_{\mathrm{C.L.D.P}}^{n+1} \circ \partial_{\mathrm{C.L.D.P}}^{n} \Big((f^1,g^1),(f^2,g^2),\ldots,(f^n,g^n)\Big)=(-1)^{n-1} \Big(([w_1,f^1]_{\mathrm{NR}},-[w_1,g_{n-1}^1]_{\mathrm{NR}}-[f^1,\delta_1]_{\mathrm{NR}}),\ldots,\\
			&\underbrace{([w_2,f^{i-1}]_{\mathrm{NR}}+[w_1,f^i]_{\mathrm{NR}},-[w_2,g^{i-1}]_{\mathrm{NR}}-[w_1,g^i]_{\mathrm{NR}}-[f^{i-1},\delta_2]_{\mathrm{NR}}-[f^i,\delta_1]_{\mathrm{NR}})}_{2 \leq i \leq n},\ldots,
			([w_2,f^n]_{\mathrm{NR}},-[w_2,g^n]_{\mathrm{NR}}-[f^n,\delta_2]_{\mathrm{NR}})\Big) \\
			&=\Big(([w_1,[w_1,f^1]_{\mathrm{NR}}]_{\mathrm{NR}},-[w_1,-[w_1,g^1]_{\mathrm{NR}}-[f^1,\delta_1]_{\mathrm{NR}}]_{\mathrm{NR}}-[[w_1,f^1]_{\mathrm{NR}},\delta_1]_{\mathrm{NR}}),\\
			&([w_2,[w_1,f^1]_{\mathrm{NR}}]_{\mathrm{NR}}+[w_1,[w_2,f^1]_{\mathrm{NR}}+[w_1,f^2]_{\mathrm{NR}}]_{\mathrm{NR}},-[w_2,-[w_1,g^1]_{\mathrm{NR}}-[f^1,\delta_1]_{\mathrm{NR}}]_{\mathrm{NR}}\\
			&-[w_1,-[w_2,g^1]_{\mathrm{NR}}-[w_1,g^2]_{\mathrm{NR}}-[f^1,\delta_2]_{\mathrm{NR}}-[f^2,\delta_1]_{\mathrm{NR}}]_{\mathrm{NR}}-[[w_1,f^1]_{\mathrm{NR}},\delta_2]_{\mathrm{NR}}-[[w_2,f^1]_{\mathrm{NR}}+[w_1,f^2]_{\mathrm{NR}},\delta_1]_{\mathrm{NR}}),\ldots,\\
			&([w_2,[w_1,f^{n-1}]_{\mathrm{NR}}]_{\mathrm{NR}}+[w_1,[w_2,f^n]_{\mathrm{NR}}+[w_1,f^2]_{\mathrm{NR}}]_{\mathrm{NR}},-[w_2,-[w_1,g^{n-1}]_{\mathrm{NR}}-[f^n,\delta_1]_{\mathrm{NR}}]_{\mathrm{NR}}\\
			&-[w_1,-[w_2,g^{n-1}]_{\mathrm{NR}}-[w_1,g^n]_{\mathrm{NR}}-[f^{n-1},\delta_2]_{\mathrm{NR}}-[f^n,\delta_1]_{\mathrm{NR}}]_{\mathrm{NR}}-[[w_1,f^{n-1}]_{\mathrm{NR}},\delta_2]_{\mathrm{NR}}-[[w_2,f^{n-1}]_{\mathrm{NR}}+[w_1,f^n]_{\mathrm{NR}},\delta_1]_{\mathrm{NR}}),\\
			&([w_2,[w_2,f^n]_{\mathrm{NR}}]_{\mathrm{NR}},-[w_2,-[w_2,g^n]_{\mathrm{NR}}-[f^n,\delta_2]_{\mathrm{NR}}]_{\mathrm{NR}}-[[w_2,f^n]_{\mathrm{NR}},\delta_2]_{\mathrm{NR}}) \Big)
		\end{align*}
	}
	For the term $3 \leq i \leq n-1$ it is whriten as follow
	\small{
		\begin{align*}
			&,\ldots,\Big([w_2,[w_2,f^{i-2}]_{\mathrm{NR}}]_{\mathrm{NR}}+[w_2,[w_1,f^{i-1}]_{\mathrm{NR}}]_{\mathrm{NR}}+[w_1,[w_2,f^{i-1}]_{\mathrm{NR}}]_{\mathrm{NR}}+[w_1,[w_1,f^i]_{\mathrm{NR}}]_{\mathrm{NR}}, \\
			&-[w_2,-[w_2,g^{i-2}]_{\mathrm{NR}}-[w_1,g^{i-1}]_{\mathrm{NR}}-[f^{i-2},\delta_2]_{\mathrm{NR}}-[f^{i-1},\delta_1]_{\mathrm{NR}}]_{\mathrm{NR}}-[w_1,-[w_2;g^{i-1}]_{\mathrm{NR}}-[w_1,g^i]_{\mathrm{NR}}-[f^{i-1},\delta_2]_{\mathrm{NR}}-[f^i,\delta_1]_{\mathrm{NR}}]_{\mathrm{NR}}\\
			&-[[w_2,f^{i-2}]_{\mathrm{NR}}+[w_1,f^{i-1}]_{\mathrm{NR}},\delta_2]_{\mathrm{NR}}-[[w_2,f^{i-1}]_{\mathrm{NR}}+[w_1,f^i]_{\mathrm{NR}},\delta_1]_{\mathrm{NR}}\Big),\ldots, 
		\end{align*}
	}
	Now we comput all term, one by one, and using the graded Jacobi identity and proposition \eqref{prop5.7}. \\
	Let's start with the the first term
	\small{
		\begin{align*}
			&([w_1,[w_1,f^1]_{\mathrm{NR}}]_{\mathrm{NR}},-[w_1,-[w_1,g^1]_{\mathrm{NR}}-[f^1,\delta_1]_{\mathrm{NR}}]_{\mathrm{NR}}-[[w_1,f^1]_{\mathrm{NR}},\delta_1]_{\mathrm{NR}})=\\
			&=(0,[w_1,[w_1,g^1]_{\mathrm{NR}}]_{\mathrm{NR}}-[w_1,[f^1,\delta_1]_{\mathrm{NR}}]_{\mathrm{NR}}
			+[[w_1,f^1]_{\mathrm{NR}},\delta_1]_{\mathrm{NR}}) \\
			&=(0,[w_1,[f^1,\delta_1]_{\mathrm{NR}}]_{\mathrm{NR}}-[[w_1,f^1]_{\mathrm{NR}},\delta_1]_{\mathrm{NR}}) \\
			&=(0,0)
		\end{align*}
	}
	For the second term
	\begin{align*}
		&([w_2,[w_1,f^1]_{\mathrm{NR}}]_{\mathrm{NR}}+[w_1,[w_2,f^1]_{\mathrm{NR}}+[w_1,f^2]_{\mathrm{NR}}]_{\mathrm{NR}},-[w_2,-[w_1,g^1]_{\mathrm{NR}}-[f^1,\delta_1]_{\mathrm{NR}}]_{\mathrm{NR}}\\
		&-[w_1,-[w_2,g^1]_{\mathrm{NR}}-[w_1,g^2]_{\mathrm{NR}}-[f^1,\delta_2]_{\mathrm{NR}}-[f^2,\delta_1]_{\mathrm{NR}}]_{\mathrm{NR}}-[[w_1,f^1]_{\mathrm{NR}},\delta_2]_{\mathrm{NR}}-[[w_2,f^1]_{\mathrm{NR}}+[w_1,f^2]_{\mathrm{NR}},\delta_1]_{\mathrm{NR}}) \\
		&=(0,[w_2,[w_1,g^1]_{\mathrm{NR}}]_{\mathrm{NR}}+[w_2,[f^1,\delta_1]_{\mathrm{NR}}]_{\mathrm{NR}}
		+[w_1,[w_2,g^1]_{\mathrm{NR}}]_{\mathrm{NR}}+[w_1,[w_1,g^2]_{\mathrm{NR}}]_{\mathrm{NR}}+[w_1,[f^1,\delta_2]_{\mathrm{NR}}]_{\mathrm{NR}}+[w_1,[f^2,\delta_1]_{\mathrm{NR}}]_{\mathrm{NR}} \\
		&-[[w_1,f^1]_{\mathrm{NR}},\delta_2]_{\mathrm{NR}}-[[w_2,f^1]_{\mathrm{NR}},\delta_1]_{\mathrm{NR}}-[[w_1,f^2]_{\mathrm{NR}},\delta_1]_{\mathrm{NR}}) \\
		&=(0,[w_2,[f^1,\delta_1]_{\mathrm{NR}}]_{\mathrm{NR}}
		+[w_1,[f^1,\delta_2]_{\mathrm{NR}}]_{\mathrm{NR}}+[w_1,[f^2,\delta_1]_{\mathrm{NR}}]_{\mathrm{NR}} -[[w_1,f^1]_{\mathrm{NR}},\delta_2]_{\mathrm{NR}}-[[w_2,f^1]_{\mathrm{NR}},\delta_1]_{\mathrm{NR}}-[[w_1,f^2]_{\mathrm{NR}},\delta_1]_{\mathrm{NR}}) \\
		&=(0,0)
	\end{align*}
	For term $3 \leq i \leq n-1$ 
	\begin{align*}
		&\Big([w_2,[w_2,f^{i-2}]_{\mathrm{NR}}]_{\mathrm{NR}}+[w_2,[w_1,f^{i-1}]_{\mathrm{NR}}]_{\mathrm{NR}}+[w_1,[w_2,f^{i-1}]_{\mathrm{NR}}]_{\mathrm{NR}}+[w_1,[w_1,f^i]_{\mathrm{NR}}]_{\mathrm{NR}}, \\
		&-[w_2,-[w_2,g^{i-2}]_{\mathrm{NR}}-[w_1,g^{i-1}]_{\mathrm{NR}}-[f^{i-2},\delta_2]_{\mathrm{NR}}-[f^{i-1},\delta_1]_{\mathrm{NR}}]_{\mathrm{NR}}-[w_1,-[w_2,g^{i-1}]_{\mathrm{NR}}-[w_1,g^i]_{\mathrm{NR}}-[f^{i-1},\delta_2]_{\mathrm{NR}}-[f^i,\delta_1]_{\mathrm{NR}}]_{\mathrm{NR}}\\
		&-[[w_2,f^{i-2}]_{\mathrm{NR}}+[w_1,f^{i-1}]_{\mathrm{NR}},\delta_2]_{\mathrm{NR}}-[[w_2,f^{i-1}]_{\mathrm{NR}}+[w_1,f^i]_{\mathrm{NR}},\delta_1]_{\mathrm{NR}}\Big)=\\
		&=(0,[w_2,[w_2,g^{i-2}]_{\mathrm{NR}}]_{\mathrm{NR}}+[w_2,[w_1,g^{i-1}]_{\mathrm{NR}}]_{\mathrm{NR}}+[w_2,[f^{i-2},\delta_2]_{\mathrm{NR}}]_{\mathrm{NR}}+[w_2,[f^{i-1},\delta_1]_{\mathrm{NR}}]_{\mathrm{NR}}\\
		&+[w_1,[w_2,g^{i-1}]_{\mathrm{NR}}]_{\mathrm{NR}}+[w_1,[w_1,g^i]_{\mathrm{NR}}]_{\mathrm{NR}}+[w_1,[f^{i-1},\delta_2]_{\mathrm{NR}}]_{\mathrm{NR}}+[w_1,[f^i,\delta_1]_{\mathrm{NR}}]_{\mathrm{NR}} \\
		&-[[w_2,f^{i-2}]_{\mathrm{NR}},\delta_2]_{\mathrm{NR}}-[w_1,f^{i-1}]_{\mathrm{NR}},\delta_2]_{\mathrm{NR}}-[[w_2,f^{i-1}]_{\mathrm{NR}},\delta_1]_{\mathrm{NR}}-[w_1,f^i]_{\mathrm{NR}},\delta_1]_{\mathrm{NR}}) \\
		&=(0,0)
	\end{align*}
	Simliarly for the rest of terms, then we obtain
	\begin{equation*}
		\partial_{\mathrm{C.L.D.P}}^{n+1} \circ \partial_{\mathrm{C.L.D.P}}^{n} \Big((f^1,g^1),(f^2,g^2),\ldots,(f^n,g^n)\Big)=((0,0),(0,0),\ldots,(0,0))
	\end{equation*}
	This complete the proof.
\end{proof}
By the adjoint representation, we obtain a cochains complex $(\mathcal{C}^{\star}_{C.L.D}(L;L),\partial)$ we denoted the set of closed $n$-cochains by $\mathbb{Z}^n_{C.L.D}(L;L)$ and the set of exact $n$-cochains by $\mathbb{Z}^n_{C.L.D}(L;L)$. We define the corresponding cohomology group by 
\begin{equation*}
	\mathbb{H}^n_{C.L.D}(L;L)=\mathbb{Z}^n_{C.L.D}(L;L) / \mathbb{Z}^n_{C.L.D}(L;L)
\end{equation*}

\noindent {\bf Acknowledgment:}
The authors would like to thank the referee for valuable comments and suggestions on this article.

\end{document}